\documentclass[sn-mathphys,Numbered]{sn-jnl}

\usepackage{graphicx}%
\usepackage{multirow}%
\usepackage{amsmath,amssymb,amsfonts}%
\usepackage{amsthm}%
\usepackage{mathrsfs}%
\usepackage[title]{appendix}%
\usepackage{xcolor}%
\usepackage{textcomp}%
\usepackage{manyfoot}%
\usepackage{booktabs}%
\usepackage{algorithm, algpseudocode}
\usepackage{algorithmicx}%
\usepackage{algpseudocode}%

\makeatletter
\def\plist@algorithm{Alg.\space}
\makeatother

\usepackage{listings}%
\usepackage{hyperref}%
\usepackage{cleveref}
\usepackage{setspace}
\usepackage{subcaption}
\usepackage{float}
\usepackage{amssymb}

\newcommand{\Matlab}{\textsc{Matlab }}
\newcommand{\mlin}[1]{\mbox{\tt{#1}}}

\DeclareMathOperator{\rank}{rank}

\DeclareMathOperator{\vec2}{vec}

\newtheorem{theorem}{Theorem}
\newtheorem{lemma}[theorem]{Lemma}%

\begin{document}

\title[Article Title]{Stability Improvements for Fast Matrix Multiplication}


\author*[1]{\fnm{Vermeylen} \sur{Charlotte}}\email{charlotte.vermeylen@kuleuven.be}

\author[1]{\fnm{Van Barel} \sur{Marc}}\email{marc.vanbarel@kuleuven.be}

\affil[1]{\orgdiv{Department of Computer Science}, \orgname{KU Leuven}, \orgaddress{\street{Celestijnenlaan 200 A}, \city{Heverlee}, \postcode{3001}, \country{Belgium}}}


\abstract{We implement an Augmented Lagrangian method to minimize a constrained least-squares cost function designed to find polyadic decompositions of the matrix multiplication tensor. We use this method to obtain new discrete decompositions and parameter families of decompositions. Using these parametrizations, faster and more stable matrix multiplication algorithms can be discovered.}




\maketitle

\section{Introduction}\label{sec1}
In this paper, we propose a new numerical optimization algorithm to discover faster and more stable matrix multiplication algorithms. To speed up matrix multiplication, we want to find a \emph{(canonical) polyadic decomposition ((C)PD)} of the so called \textit{matrix multiplication tensor (MMT)} \cite{Strassen1969}. This tensor is defined by the bilinear equations corresponding to multiplying two matrices $A \in \mathbb{R}^{m\times p}$ and $B \in \mathbb{R}^{p \times n}$:
\begin{equation} \label{eq:def_MMT}
\begin{split}
\text{vec} ((AB)^\top) = T_{mpn} \cdot_1 \text{vec}(A) \cdot_2 \text{vec}(B),
\end{split}
\end{equation}
where $\vec2\left( \cdot \right)$ denotes the column-wise vectorization operator, $T_{mpn}$ is the MMT of dimension $mp \times pn \times mn$, and $\cdot_1$ and $\cdot_2$ denote the multiplication along the first and second dimension of $T_{mpn}$ respectively. Because of this definition, $T_{mpn}$ is a sparse tensor consisting of $mpn$ ones: $T_{mpn}(i_1 + (i_2-1)m,j_1 + (j_2-1)p,k_1 + (k_2-1)n) = 
1 \text{ if } i_1 = k_2, i_2 = j_1$, and $j_2 = k_1$, and zero else,
for all $i_1,k_2 = 1 \dots m$, $i_2,j_1 = 1 \dots p$, and $j_2, k_1 = 1 \dots n$.
A PD of $T_{mpn}$ of \emph{rank} $R$ is a decomposition of $T_{mpn}$ into $R$ rank-1 tensors:
\begin{equation} \label{eq:CPD}
\begin{split}
T_{mpn} = \sum_{r=1}^R u_r \otimes v_r \otimes w_r,
\end{split}
\end{equation}
where $\otimes$ denotes the tensor product and $u_r,v_r$, and $w_r$ are vectors in $\mathbb{R}^{mp}$, $\mathbb{R}^{pn}$, and $\mathbb{R}^{mn}$ respectively. These vectors are collected in three \emph{factor matrices} $U$, $V$, and $W$:
$U = \left[ u_{1} \cdots u_{R} \right]$, $V = \left[ v_{1} \cdots v_{R} \right]$, and $W = \left[ w_{1} \cdots w_{R} \right]$.
A PD of rank $R$ is written more short as $\mathrm{PD}_R$ and a PD of rank $R$ of a tensor $T$ as $\mathrm{PD}_R(T)$. The minimal $R$ for which a PD of a certain tensor exists is called the rank $R^*$ of this tensor or $R^*(T)$ and a PD of this rank a CPD or $\mathrm{PD}_{R^*}(T)$. With a $\mathrm{PD}_R$ \eqref{eq:CPD} of $T_{mpn}$ and using \eqref{eq:def_MMT}, we obtain the following \emph{basis algorithm} to multiply any two matrices $A \in \mathbb{R}^{m\times p}$ and $B \in \mathbb{R}^{p \times n}$:
\begin{align} \label{eq:vecAB_TM}
\text{vec}\left( \left(AB \right)^\top \right) &= \sum_{r=1}^R \langle u_r, \text{vec}(A) \rangle \langle v_r,  \text{vec}(B) \rangle w_r  \nonumber \\
\Leftrightarrow~  AB &= \sum_{r=1}^R \text{trace} \left( U_r A^\top \right) \text{trace} \left(V_r  B^\top \right) W_r^\top,
\end{align}
where $\langle \cdot \rangle$ denotes the standard vector inner product and $U_r = \text{reshape}(u_r,[m,p])$, $V_r = \text{reshape}(v_r,[p,n])$, and $W_r = \text{reshape}(w_r,[n,m])$, for all $r=1 \dots R$. The number of multiplications between (linear combinations of) elements of $A$ and $B$ in \eqref{eq:vecAB_TM}  is reduced to $R < mpn$, where the upper bound holds for standard matrix multiplication. Such multiplications are called ‘active' because, when these algorithms are applied recursively, they determine the asymptotic complexity of matrix multiplication \cite{Landsberg_complexity}.

However, $R^*(T_{mpn})$ is not known for most combinations of $(m,p,n)$. Counterexamples are $T_{222}$, for which the set of PDs is completely understood and $R^*$ is known to be seven \cite{DeGroote1978b}, and $T_{223}$, for which the rank is known to be 11 \cite{ALEKSEYEV_1985}. We call $\mathrm{PD}_{7}(T_{222})$ originally discovered by Strassen $\mathrm{PD}_{\mathrm{Strassen}}$ \cite{Strassen1969}. For other combinations of $m$, $p$, and $n$, lower bounds on the rank exist but none of these lower bounds are attained in practice. For example, the lower bound on $R^*(T_{333})$ is proven to be 19 \cite{Blaser2003}, and the upper bound is 23 \cite{Laderman1976, Ballard2019, Smirnov2013, new_ways_33_2021}. We call the lowest rank for which an exact PD of $T_{mpn}$ is known in the literature $\tilde{R}(T_{mpn})$. For an overview of $\tilde{R}(T_{mpn})$ for various $m$, $p$, and $n$, we refer to \cite[Table 1]{Smirnov2013} and \cite[Fig. 1]{DeepMind_2022}. Remark that the rank of $T_{mpn}$ is equal to the rank of $T_{pnm}$ or any other permutation of $m$, $p$, and $n$. This is true because we can obtain a $\mathrm{PD}_R$ of $T_{pnm}$ and $T_{nmp}$ by cyclically permuting the factors in the PD:
\begin{align*}
T_{pnm} = \sum_{r=1}^R v_r \otimes w_r \otimes u_r, && T_{nmp} = \sum_{r=1}^R w_r \otimes u_r \otimes v_r,
\end{align*}
and additionally, it can be shown that
$T_{mnp} = \sum_{r=1}^R \vec2 \left(W_r^\top \right) \otimes \vec2(V_r^\top) \otimes \vec2(U_r^\top)$, and similarly for $T_{npm}$ and $T_{pmn}$.

To find a $\mathrm{PD}_R$ of $T_{mpn}$, we use the following least-squares cost function:
\begin{equation} \label{eq: unconstr_LSQ_min}
\min_x  \underbrace{\frac{1}{2} \left\lVert  F(x)- \text{vec} \left(T_{mpn} \right) \right\rVert^2}_{=: f(x)},
\end{equation}
where $\lVert \cdot \rVert $ denotes the $\mathrm{L}_2$-norm, and $F(x)$ is a vector function defined as
\begin{align} \label{eq:F(x)}
F\left(x \right) &:=  \mathrm{vec} \left( \sum_{r=1}^R u_r \otimes v_r \otimes w_r \right), & x := \begin{bmatrix}
\vec2(U)^\top & \vec2(V)^\top & \vec2(W)^\top  
\end{bmatrix}^\top.
\end{align}
Most standard optimization algorithms fail to converge to a global minimum of \eqref{eq: unconstr_LSQ_min} and, in general, finding the CPD of a tensor is NP-hard \cite{tensor_rank_NP_hard_1989}. One of the causes is the nonuniqueness of a $\mathrm{PD}_R$ and consequently the minima are not isolated. This can be resolved, as long as $R$ is smaller than a certain function of $m$, $p$, and $n$, by using manifold optimization \cite{RTR_CPD_Nick_2018}. One of the important properties of manifolds is that, e.g., two PDs on the manifold may not generate the same tensor\footnote{For more information on manifolds and optimization on manifolds, we refer to \cite{intro_manifolds,matrix_manifolds}.}. The $\mathrm{PD}$s of $T_{mpn}$ do not satisfy the criterion on the rank described in \cite{RTR_CPD_Nick_2018}, and furthermore, have additional continuous invariances compared to generic PDs. Some of these \emph{invariance transformations} were discovered by De Groote \cite{DeGroote1978a}, and others by parametrizing discrete PDs \cite{params_R23_1986,new_ways_33_2021}. Because these additional invariances are not completely understood, nor fully discovered as we show different new parametrizations in this paper, it is not known if the set of $\mathrm{PD}$s of $T_{mpn}$ of rank $R$ belongs to a manifold.

Another difficulty is that $T_{mpn}$ is known to admit \emph{border rank} PDs (BRPDs) \cite{Bini1980approximate,schonhagen_partial_total_MM_1981,Landsberg2015a}. BRPDs are approximate solutions of \eqref{eq: unconstr_LSQ_min} with elements that grow to infinity when the error of the approximation goes to zero. The minimal rank for which a BRPD of a certain tensor exists is called the border rank $R_b$ of this tensor. The border rank is in general smaller than the canonical rank. The border rank for different combinations of $m$, $p$, and $n$, was, e.g., investigated in \cite{Smirnov2013} using a heuristic alternating least-squares (ALS) method.

Since the total number of elements of the factor matrices grows rapidly with the size of the matrices, solving \eqref{eq: unconstr_LSQ_min} is only feasible for relatively small values of $m$, $p$, and $n$, e.g., $m,p,n \leq 5$. The resulting $\mathrm{PD}_R$ can be applied recursively to, e.g., multiply matrices of size $m^k \times p^k$ and $p^k \times n^k$. Such a recursive algorithm requires $\mathcal{O} \left( cR^k\right)$ operations, where the constant $c$ depends on the number of scalar multiplications and nonzeros in the factor matrices. When $k$ is sufficiently large and $c$ sufficiently small, fewer operations are required compared to standard matrix multiplication. To decrease the constant $c$, we search for sparse PDs, with elements in a discrete set $\lbrace{-1,0,1 \rbrace}$. This set can eventually be extended with powers of 2, since this is not a costly operation when implemented in hardware. Such PDs result in practical FMM algorithms. 
Thus, a numerical solution of \eqref{eq: unconstr_LSQ_min} still has to be transformed into a sparse and discrete PD. The \emph{invariance transformations} discovered in \cite{DeGroote1978a}, which are further discussed in \Cref{sec:inv_transf}, can be used to obtain such PDs \cite{Tichavsky2017}. As is known from \cite{berger2022equivalent} however, not all PDs of $T_{mpn}$ are discretizable in this way.

We propose to solve \eqref{eq: unconstr_LSQ_min} in combination with a bound constraint on the elements of the factor matrices to prevent the convergence to BRPDs and additionally we propose different equality constraints, which are discussed in \Cref{sec:constraints}, to obtain discrete PDs and parameter families of PDs. To solve this constrained optimization problem, we implement an Augmented Lagrangian (AL) method, given in \Cref{sec:ALM}, and use the Levenberg-Marquardt (LM) method to solve the subproblem of the minimization to $x$. The LM method is very suitable for least-squares problems with a singular Jacobian matrix, such as \eqref{eq: unconstr_LSQ_min} \cite{Nocedal1999,madsen2004}. We discuss the LM method in more detail in \Cref{sec:LM}.

Constrained optimization has already been used multiple times in the literature to find sparse, discrete solutions of \eqref{eq: unconstr_LSQ_min}. For example, the method in \cite{Smirnov2013} uses an ALS method in combination with a quadratic penalty (QP) term to bound the factor matrix elements to a certain interval. 
On the other hand, the method in \cite{Tichavsky2017} implements an LM method with a constraint on the norm of the factor matrices using a Lagrange parameter. 
However, both methods were, e.g., not able to obtain exact PDs for $T_{444}$, rank 49, whereas such PDs are known to exist. The method that we propose can be considered as a combination of the methods in \cite{Tichavsky2017,Smirnov2013}; we use the LM method and Lagrange parameters as in \cite{Tichavsky2017} and a constraint to bound the elements of the factor matrices as in \cite{Smirnov2013}. We also tested the LM method in combination with a QP term but the advantage of the AL method is that the constraints are satisfied more accurately during optimization. With this method we are, e.g., able to obtain new (discrete) PDs for $T_{444}$, rank 49.

A last element to consider is the stability of practical FMM algorithms \cite{Ballard2016_stab}. The stability of an FMM algorithm is determined by a pre- and stability factor. We give the definition of these factors in \Cref{sec:stab}. We improve the stability by reducing both factors for several combinations of $m$, $p$, and $n$, and for some of these combinations, we additionally reduce the number of nonzeros and consequently speed up matrix multiplication by reducing the constant $c$ in the asymptotic complexity. To improve the stability, we search for sparse parameter families of PDs and optimize the stability factor within the family. In general, the sparser the family, the lower the prefactor.

For $T_{333}$, rank 23, several parameter families are known in the literature, e.g., two parameter families with one and three parameters respectively were found using a heuristic ALS method  \cite{params_R23_1986}, and up to 17 parameters were introduced in discrete PDs obtained using an SAT solver \cite{new_ways_33_2021}.
We extend these parametrizations with a two-parameter family of a non-discretizable PD which is, up to our knowledge, the first parametrization of such a PD. These PDs and corresponding parametrizations might not be useful in practical FMM algorithms but they might help us understand the set of all (C)PDs of $T_{mpn}$. Furthermore, we are, up to our knowledge, the first to propose parameter families for other combinations of $m,p$, and $n$, such as 
$(4,4,4)$, $(3,4,5)$ and permutations. Using these parametrizations, the stability of FMM algorithms can be investigated and improved using a simple numerical optimization routine in {\Matlab}. We round the obtained numerical parameters to the nearest power of two to obtain practical PDs.

Because of its definition, $T_{mpn}$ is a structured and more specifically cyclic symmetric (CS) tensor when $m:=p:=n$. This means that $T_{mmm}(i,j,k) = T_{mmm}(j,k,i)=T_{mmm}(k,i,j)$, for all $i,j,k=1\dots m^2$. As discussed in \Cref{sec:CS}, this structure can be used in the optimization problem to reduce the number of unknowns. We denote by $S$ the number of symmetric rank-1 tensors in a CS PD and by $T$ the number of CS pairs of asymmetric rank-1 tensors. The rank then equals $R=S+3R$. Different CS, sparse, and discrete PDs of rank 7 for $T_{222}$ and of rank 23 for $T_{333}$ are known in the literature \cite{Chiantini2019,Ballard2019}. These PDs were found using an ALS method and further investigated using algebraic geometry and group theory. However, no parametrizations for these PDs were given. Additionally, no CS PDs for $T_{444}$ are known to us, except for the PD obtained when using $\mathrm{PD}_{\mathrm{Strassen}}$, which satisfies  $S:=1$ and $T:=2$, recursively. This recursive PD is denoted by $\mathrm{PD}_{\mathrm{Strassen}}^{\mathrm{rec}}$ and can be shown to satisfy $S:=1$ and $T:=16$. Another CS PD with $S:=4$ and $T:=1$ exists for $T_{222}$ but this PD contains more nonzeros than $\mathrm{PD}_{\mathrm{Strassen}}$ and thus is not useful in practice \cite{Chiantini2019}. However, when this PD is used recursively, we obtain another CS $\mathrm{PD}_{49} \left(T_{444} \right)$ with $S:=16$ and $T:=11$. 
In \Cref{sec:param_and_stab_improv}, we give a five-parameter familiy of CS $\mathrm{PD}_{49} \left(T_{444} \right)$s for $S:=13$ and $T:=12$, obtained with the method described in \Cref{sec:ALM}.  

This paper is organised as follows. In \Cref{sec:prelim}, the preliminaries for the constrained optimization problem and the proposed numerical optimization algorithm are given. The AL optimization method is then given in \Cref{sec:ALM}. Lastly in \Cref{sec:results}, the different parametrizations we discovered are shown, together with the values of the parameters for which an improvement of the stability is attained. 

\section{Preliminaries} \label{sec:prelim}

In this preliminary section, we first discuss the well known invariance transformations that hold for all PDs of $T_{mpn}$ and we give a brief overview of what is known about the use of these invariance transformations to obtain sparse and discrete PDs of $T_{mpn}$. Afterwards, BRPDs and the definition of the stability of FMM algorithms are briefly discussed. Then, we give a first result about the link between the rank of the Jacobian matrix of \eqref{eq: unconstr_LSQ_min} and the invariance transformations. Lastly, we discuss the LM method and the quadratic penalty method because they form the basis for the AL method that is proposed in the next section.

\subsection{Invariance Transformations} \label{sec:inv_transf}
The representation of a PD via factor matrices is not unique. For example when permuting the rank-1 tensors, and consequently the columns of the factor matrices according to
\begin{equation*}
\lbrace U,V,W \rbrace \rightarrow \lbrace U',V',W' \rbrace,
\end{equation*}
where $ u'_r= u_{\sigma(r)}$, for all $r=1\dots R$, and $\sigma = \left( \begin{matrix}
\sigma(1) & \sigma(2) & \sigma(3) & \dots & \sigma(R)
\end{matrix} \right)$ is an element of the permutation group of $R$ elements, then the same tensor is obtained:
\begin{equation*}
\sum_{r=1}^R u_r \otimes v_r\otimes w_r = \sum_{r=1}^R u_r' \otimes v_r'\otimes w_r'.
\end{equation*}
Additionally, because of the multilinearity of the tensor product, the columns of each rank-1 tensor can be scaled as follows: $\alpha_r u_r, \beta_r v_r, \frac{1}{\alpha_r \beta_r} w_r$, for all $r=1 \dots R$, and $\alpha_r,\beta_r$ in $\mathbb{R}_0$, without changing the rank-1 tensors because
\begin{equation*}
\begin{split}
\alpha_r u_r \otimes  \beta v_r  \otimes \frac{1}{\alpha_r \beta_r} w_r = u_r \otimes  v_r  \otimes w_r. \\
\end{split}
\end{equation*}
These two invariances hold for the PDs of all tensors. Remark that only the scaling invariance is a continuous invariance transformation (of dimension $2R$).

For PDs of $T_{mpn}$, additional invariances hold \cite{DeGroote1978a}. More specifically, it is well known that these PDs are invariant under the following \emph{PQR-transformation}:
\begin{equation} \label{eq:mult_inv}
\begin{split}
u_r' \leftarrow \text{vec} \left( P U_r Q^{-1} \right), \quad v_r' \leftarrow \text{vec} \left( Q V_r R^{-1} \right),  \quad w_r' \leftarrow \text{vec} \left( R W_r P^{-1} \right), \\
\end{split}
\end{equation}
for $r=1 \dots R$ and where $P \in GL(m)$, $Q, \in GL(p)$, $R \in GL(n)$, where $GL(i)$ is the \emph{general linear group} of invertible matrices in $\mathbb{R}^{i \times i}$. This invariance can easily be proven using the fact that $C = P^{-1} \left( PAQ^{-1} \right) \left( QBR^{-1} \right) R$. Additionally, when $m:=n:=p$, PDs of $T_{mmm}$ are invariant under the following \emph{transpose-transformation}:
\begin{equation*}
\begin{split}
u_r' \leftarrow \text{vec} \left( V_r^\top \right) , \quad v_r' \leftarrow \text{vec} \left( U_r^\top \right) , \quad w_r' \leftarrow \text{vec} \left( W_r^\top \right),\\
\end{split}
\end{equation*}
for $r=1 \dots R$, which can be proven using the fact that $C = \left( B^\top A^\top \right)^\top$. Lastly, still for $m:=n:=p$, the PDs are invariant under cyclic permutation of the factor matrices:
\begin{equation*}
\begin{split}
\sum_{r=1}^R u_r \otimes v_r\otimes w_r = \sum_{r=1}^R v_r \otimes w_r\otimes u_r = \sum_{r=1}^R w_r \otimes u_r\otimes v_r,
\end{split}
\end{equation*}
which holds because $T_{mmm}$ is a CS tensor.

Remark that for PDs of $T_{mpn}$, only the PQR-invariance is a continuous invariance and it overlaps with the scaling invariance when $P=aI$, $Q=bI$, and $R=cI$, where $I$ is the identity matrix of the appropriate size and $a,b$, and $c$ are scaling factors in $\mathbb{R}_0$, because in this case the PQR-transformation scales the columns of the factor matrix $U$ with $\frac{a}{b}$, all columns of $V$ with $\frac{b}{c}$, and all columns of $W$ with $\frac{c}{a}$, and indeed the product of these factors equals one as for the scaling invariance. Thus, the combination of the scaling and PQR-transformation has at most dimension $2R + m^2 + p^2 + n^2 - 3$ for all PDs of $T_{mpn}$.

We call the combination of the invariance transformations described above \emph{inv-tranformations}. Two PDs of $T_{mpn}$ that can be obtained using inv-transformations are called \emph{inv-equivalent}. Only for $T_{222}$, it is known that all PDs are inv-equivalent \cite{DeGroote1978b}.

\subsubsection{Discretization using Invariance Transformations} \label{sec:discret_inv_trans}
We have motivated in the introduction that to obtain a practical FMM algorithm, sparse and discrete PDs are required. If we obtain a non-discrete PD using numerical optimization, we can attempt to obtain a discrete PD with the inv-transformations \cite{Tichavsky2017}. This means that the numerical PD has to be inv-equivalent to a discrete PD. However, Guillaume Berger discovered, based on experiments performed with the method of Tichavsk\'y \cite{Tichavsky2017}, that two numerical PDs are in general not inv-equivalent for most combinations of $m$, $p$, and $n$, except, e.g., for $T_{222}$, for which all PDs are inv-equivalent to $\mathrm{PD}_{\mathrm{Strassen}}$ and consequently discretizable \cite{Strassen1969,DeGroote1978b}. He additionally investigated necessary conditions for a PD to be discretizable in this way \cite{berger2022equivalent}. For example, the trace of $U_r V_r W_r$ has to be discrete, for all $r=1 \dots R$, because the inv-transformations are trace invariant. Similarly, the (sorted) ranks of $U_rV_rW_r$ are invariant under these transformations and thus these ranks can be used as another necessary condition for the inv-equivalence of PDs.

\subsection{Cyclic Symmetry} \label{sec:CS}
As discussed already in the introduction, $T_{mpn}$ is a CS tensor when $m:=p:=n$. We can make use of this property for constructing a PD by requiring that the rank-1 tensors are symmetric or occur in CS pairs:
\begin{equation*}
T_{mmm} = \sum_{s=1}^S a_s \otimes a_s \otimes a_s +\sum_{t=1}^T b_t \otimes d_t \otimes c_t +\sum_{t=1}^T c_t \otimes b_t \otimes d_t + \sum_{t=1}^T d_t \otimes c_t \otimes b_t,
\end{equation*}
where $a_s,b_t,d_t$, and $c_t$ are vectors in $\mathbb{R}^{m^2}$, for all $s =1\dots S$ and $t=1\dots T$. The parameters $S$ and $T$ determine respectively the number of symmetric and CS pairs of asymmetric rank-1 tensors. The rank then equals $R = S + 3T$. If we define the matrices
$A := \left[ a_{1} \cdots a_{S} \right]$, 
$B := \left[ b_{1} \cdots b_{T} \right]$, $C := \left[ c_{1} \cdots c_{T} \right]$, and $D := \left[ d_{1} \cdots d_{T} \right]$,
the factor matrices can be written in function of these matrices:
\begin{align*}
U& = \begin{bmatrix} A & B &  C & D \end{bmatrix}, && V= \begin{bmatrix} A & D & B & C \end{bmatrix}, && W = \begin{bmatrix} A & C & D & B \end{bmatrix}.
\end{align*}
Consequently the number of unknowns in \eqref{eq: unconstr_LSQ_min} is reduced by a factor 3, which makes this CS structure very useful for larger problem parameters, e.g., $m:=p:=n:=4$. A disadvantage may be that by restricting the search space we might not be able to find the most sparse or stable PDs of $T_{mmm}$.

\subsection{Border Rank Decompositions}

An approximate or BRPD of $T_{mpn}$ of rank $R' \geq R_b$, where $R_b$ is the border rank, can be written in function of a parameter $\varepsilon<1$ as \cite{schonhagen_partial_total_MM_1981}:
\begin{equation} \label{eq:def_approx_decomp}
\frac{1}{\varepsilon^h} \left( \sum_{r=1}^{R'}\left( \sum_{i=0}^{h} \varepsilon^i u_{i,r} \right) \otimes \left( \sum_{i=0}^{h} \varepsilon^i v_{i,r} \right) \otimes \left( \sum_{i=0}^{h} \varepsilon^i w_{i,r} \right)  \right) = T_{mpn} +  \sum_{i=1}^{d} \varepsilon^i E_i,
\end{equation}
where $u_{i,r} \in \mathbb{R}^{mp}$, $v_{i,r} \in \mathbb{R}^{np}$, $w_{i,r} \in \mathbb{R}^{mn}$, and $E_{i} \in \mathbb{R}^{mp \times pn \times mn}$. The parameter $h$ is called the order of the approximate PD and $d$ the error degree. As can be seen from \eqref{eq:def_approx_decomp}, the error will go to zero when $\varepsilon$ goes to zero. However, in this case the norm of the approximate PD will go to infinity, because of the term $1/\varepsilon^h$. Furthermore, all elements of the tensor on the right hand side have norm $\mathcal{O}(1)$ when $\varepsilon \rightarrow 0$, and thus the large elements on the left hand side have to annihilate, which makes such PDs not numerically stable and thus not useful in practical FMM algorithms.

\subsection{Stability} \label{sec:stab}

To investigate the stability of practical FMM algorithms, the following upper bound on the infinity norm between $\hat{C}$, generated when using a basis algorithm to multiply matrices in $\mathbb{R}^{m \times p}$ and $\mathbb{R}^{p \times n}$ $k$ times recursively, and $C=AB$, where $A \in \mathbb{R}^{m^k \times p^k}$ and $B \in \mathbb{R}^{p^k \times n^k}$, is used \cite[Theorem 3]{Ballard2016_stab}:
\begin{equation*}
\lVert \hat{C}-C \rVert_{\infty} \leq \left(1 + qk \right)e^k \lVert A \rVert_{\infty} \lVert B \rVert_{\infty} \epsilon + \mathcal{O} \left( \epsilon^2 \right),
\end{equation*}
where $\epsilon$ is the machine precision, and $q$ and $e$ are the pre- and stability factor of the basis algorithm respectively and are defined as follows
\begin{equation} \label{eq:e_q}
\begin{split}
q &:= \max_i \left( \lVert W(i,:) \rVert_0  + \max_r \Big(\left( \lVert  u_r \rVert_0  + \lVert  v_r \rVert_0 \right) \lVert  w_r(i) \rVert_0 \Big)\right), \\
e &:= \max_i \sum_{r=1}^R \left\lVert u_r \right\rVert_1  \lVert v_r \rVert_1  \lvert w_r(i) \rvert, \quad i= 1 \dots mn.
\end{split}  
\end{equation}
As can be seen from their definition, $e$ and $q$ are highly non-convex because of the ${\ell}_1$- and $\ell_0$-norm. Since the $\ell_0$-norm determines directly the number of nonzeros, the most sparse PDs in general also have the lowest prefactor. To improve the stability, we search for sparse parameter families of PDs and minimize the stability factor within these families. This is in general not computationally expensive because the number of parameters is significantly lower than the number of elements in the factor matrices.

Remark that $e$ and $q$ are not invariant under permutation of the factor matrices, and thus the most stable PD for $T_{mpn}$ is in general not the most stable PD for $T_{pnm}$ or any other permutation of $m$, $p$, and $n$.

\subsection{Jacobian Matrix} \label{sec:Jacobian}
The Jacobian matrix $J(x)$ of $F(x)$ is the $(mpn)^2 \times R(mp+np+mn)$ matrix containing the first order partial derivatives of $F(x)$, which can be written as
\begin{equation*}
\begin{split}
J(x) = &\left[ \begin{matrix}
w_1 \otimes v_1 \otimes I_{mp} && w_2 \otimes v_2 \otimes I_{mp} &&\cdots &&  w_r \otimes v_r \otimes I_{mp} \end{matrix} \right. \\
& \left. \begin{matrix} w_1 \otimes I_{pn} \otimes u_1 && w_2 \otimes I_{pn} \otimes u_2 && \cdots &&  w_r \otimes I_{pn} \otimes u_r \end{matrix} \right. \\
 & \left. \begin{matrix} I_{mn} \otimes v_1 \otimes u_1 &&I_{mn} \otimes v_2 \otimes u_2 && \dots && I_{mn} \otimes v_r \otimes u_r
\end{matrix} \right],
\end{split}
\end{equation*}
where $I_{mp}$ is the identity matrix of size $mp$.
The gradient of $f(x)$ is then $\nabla f(x) = J(x)^\top F(x)$. The Jacobian matrix plays an important role in the numerical optimization algorithm described in \Cref{sec:ALM} but also in discovering new parametrizations as discussed further in \Cref{sec:results}. 

We noticed that for a certain combination of $m$, $p$, and $n$, the rank of $J(x^*)$ differs for different solutions $x^*$ of \eqref{eq: unconstr_LSQ_min}. If we assume that $F^{-1}(T_{mpn})$ is locally a smooth manifold at a solution $x^*$ of \eqref{eq: unconstr_LSQ_min}, then the dimension of the tangent space is equal to the rank of $J(x^*)$. Similarly, $\min \left(\mathrm{size}\left(J\left(x^*\right)\right)\right)-\rank \left(J\left(x^*\right) \right)$, where $\mathrm{size} (\cdot)$ returns a tuple containing the dimensions of a matrix, is the dimension of the solution space at $x^*$.

It is easy to see that the permutation-, scaling- and transpose-transformations of the previous section do not change the rank of $J(x)$ as these transformations only scale or permute the columns or rows of $J(x)$. The following lemma states that this also holds for the PQR-transformation \eqref{eq:mult_inv}.

\begin{lemma}[Invariance rank $J(x)$]
Let $x^*$ be such that $f(x^*)=0$ and $x^* = \begin{bmatrix} \vec2(U)^\top & \vec2(V)^\top & \vec2(W)^\top \end{bmatrix}^\top$. If $x'$ is obtained using the PQR-transformation on $U$, $V$, and $W$, then:
\begin{equation*}
\rank \left(J(x^*) \right) = \rank \left( J(x') \right).
\end{equation*} 
\end{lemma}
\begin{proof}
The PQR-transformation changes the columns of the factor matrices according to:
\begin{align*} 
u_r' \leftarrow \left( \left( Q^{-1} \right)^\top \otimes P \right) u_r, && v_r' \leftarrow \left( \left( R^{-1} \right)^\top \otimes Q \right) v_r,  && w_r' \leftarrow \left( \left(P^{-1} \right)^\top \otimes R \right) w_r.
\end{align*}
Thus the Jacobian at $x' = \begin{bmatrix} \vec2(U')^\top & \vec2(V')^\top & \vec2(W')^\top \end{bmatrix}^\top$ equals: 
\begin{gather*} 
J\left( x' \right) = 
 \left[ \begin{matrix} \left( \left(P^{-1} \right)^\top \otimes R \right) w_1 \otimes \left( \left( R^{-1} \right)^\top \otimes Q \right) v_1 \otimes I_{mp}  & \cdots & \left( \left(P^{-1} \right)^\top \otimes R \right) w_r \otimes \end{matrix} \right. \\
\left. \begin{matrix}  \left( \left( R^{-1} \right)^\top \otimes Q \right) v_r \otimes I_{mp} &~~ \left( \left(P^{-1} \right)^\top \otimes R \right) w_1 \otimes I_{pn} \otimes \left( \left( Q^{-1} \right)^\top \otimes P \right) u_1  & \cdots  \end{matrix} \right. \\
 \left. \begin{matrix} \left( \left(P^{-1} \right)^\top \otimes R \right) w_r \otimes I_{pn} \otimes \left( \left( Q^{-1} \right)^\top \otimes P \right) u_r &~~ I_{mn} \otimes \left( \left( R^{-1} \right)^\top \otimes Q \right) v_1  \end{matrix} \right.  \\
 \left. \begin{matrix}   \otimes \left( \left( Q^{-1} \right)^\top \otimes P \right) u_1 & \cdots & I_{mn} \otimes \left( \left( R^{-1} \right)^\top \otimes Q \right) v_r \otimes \left( \left( Q^{-1} \right)^\top \otimes P \right) u_r \end{matrix} \right] \\
=  \left( \left(P^{-1} \right)^\top \otimes R \right) \otimes \left( \left( R^{-1} \right)^\top \otimes Q \right) \otimes\left( \left( Q^{-1} \right)^\top \otimes P \right) \left[ \begin{matrix}  w_1 \otimes  v_1 \otimes \left( Q^\top \otimes P^{-1} \right) \end{matrix} \right. \\
 \begin{matrix}  &\cdots & w_r \otimes  v_r \otimes \left( Q^\top \otimes P^{-1} \right) &~  w_1 \otimes \left( R^\top \otimes Q^{-1} \right)   \otimes   u_1 & \cdots &   \end{matrix}\\
 \left. \begin{matrix} w_r \otimes \left( R^\top \otimes Q^{-1} \right)  \otimes  u_r &~  \left( P^\top \otimes R^{-1} \right)  \otimes v_1 \otimes  u_1 & \cdots & \left( P^\top \otimes R^{-1} \right)  \otimes  v_r \otimes u_r \end{matrix} \right] \\
=  \left( \left(P^{-1} \right)^\top \otimes R \right) \otimes \left( \left( R^{-1} \right)^\top \otimes Q \right) \otimes \left( \left( Q^{-1} \right)^\top \otimes P \right) \\
\left[ \begin{matrix}  w_1 \otimes  v_1 \otimes I_{mp} &\cdots & w_r \otimes  v_r \otimes I_{mp} & w_1 \otimes I_{pn} \otimes u_1 & \cdots   \end{matrix} \right. \\
 \left. \begin{matrix} w_r \otimes I_{pn} \otimes  u_r & I_{mn} \otimes v_1 \otimes  u_1 & \dots & I_{mn} \otimes  v_r \otimes u_r \end{matrix} \right] \text{diag} \left( \begin{matrix} Q^\top \otimes P^{-1} \end{matrix} \right. \\
\left. \begin{matrix} \cdots & Q^\top \otimes P^{-1} & R^\top \otimes Q^{-1}  \cdots & R^\top \otimes Q^{-1} &  P^\top \otimes R^{-1} & \cdots & P^\top \otimes R^{-1}  \end{matrix} \right)\\
=  \left( \left(P^{-1} \right)^\top \otimes R \right) \otimes \left( \left( R^{-1} \right)^\top \otimes Q \right) \otimes \left( \left( Q^{-1} \right)^\top \otimes P \right) J(x) \text{diag} \left( \begin{matrix} Q^\top \otimes P^{-1} \end{matrix}  \right.   \\
\left. \begin{matrix} \cdots & Q^\top \otimes P^{-1} & R^\top \otimes Q^{-1}  \cdots & R^\top \otimes Q^{-1} & P^\top \otimes R^{-1} & \cdots & P^\top \otimes R^{-1}  \end{matrix} \right).
\end{gather*}
Since both $\left( \left(P^{-1} \right)^\top \otimes R \right) \otimes \left( \left( R^{-1} \right)^\top \otimes Q \right) \otimes \left( \left( Q^{-1} \right)^\top \otimes P \right)$ and the block diagonal matrix on the right are full-rank square matrices, the rank of $J(x')$ equals the rank of $J(x)$
\end{proof}
Consequently, the rank of $J(x^*)$ can be used as another necessary condition for the inv-equivalence of PDs of $T_{mpn}$.

\subsection{Levenberg-Marquardt Method} \label{sec:LM}
The LM method, which is a damped Gauss-Newton method, is an iterative optimization method which calculates the next iterate as $x_{k+1} := x_k + \Delta x$, where $\Delta x$ is the solution of the linear system:
\begin{equation*}
(\tilde{H}(x_k)+ \mu I) \Delta x = - \nabla f(x_k),
\end{equation*}
where $\mu \geq 0 $ is a damping parameter and $\tilde{H}(x_k) := J(x_k)^\top J(x_k) $ is the Gauss-Newton approximation of the Hessian matrix of $f(x)$. The damping parameter is dynamically updated. The pseudocode of the algorithm is shown in \Cref{alg:LM} \cite{madsen2004}. 

\begin{algorithm}
\caption{LM Method for the Minimization of a Real Function $f(x)$}
\label{alg:LM}
\begin{algorithmic}[1]
\Require $x_0$, gradient tolerance $\omega^*$, step tolerance $\varepsilon$, $\mu_0$
\State Initialize $k \leftarrow 0, \nu \leftarrow 2$, $ \mu \leftarrow \mu_0 $
\While {$k< k_{\max}$}
\If {$\left\lVert \nabla f(x_k) \right\rVert < \omega^*$}
\Return $x^* \leftarrow x_{k+1}$
\EndIf
\State Solve $\left( \tilde{H}(x_k) + \mu I \right) \Delta x = -\nabla f(x_k)$
\State $x_{k+1} \leftarrow x_k + \Delta x$
\If {$\lVert \Delta x \rVert < \varepsilon (\lVert x\rVert + \varepsilon)$}
\Return $x^* \leftarrow x_{k+1}$
\EndIf
\State $\rho \leftarrow 2\left( f(x_k) -  f(x_{k+1}) \right) / \left( \Delta x^\top \left( \mu \Delta x - \nabla f(x_k) \right) \right)$
\If {$\rho > 0$}
\State $ k \leftarrow k+1$
\State $ \mu \leftarrow \mu \max\left( 1/3, 1-(2\rho-1)^3\right)$
\State $ \nu \leftarrow 2 $
\Else
\State $\mu \leftarrow \nu \mu$
\State $\nu \leftarrow 2\nu$
\EndIf
\EndWhile 
\Ensure $x_{k}$
\end{algorithmic}
\end{algorithm}

\subsection{Quadratic Penalty Method} \label{sec:quadratic_penalty}
The QP method can be used to solve a constrained optimization problem: 
\begin{equation*} 
\begin{split}
\min_x f(x), \quad \mathrm{s.t.} \quad h(x)=0,
\end{split}
\end{equation*}
where $h(x)$ is a vector function. The method adds another least-squares term to the cost function in \eqref{eq: unconstr_LSQ_min}: $f_{\beta} (x) := f(x)  + \beta \lVert h(x) \rVert^2$, 
where $\beta$ is a regularization parameter to control the relative importance of minimizing $f(x)$ versus satisfying the constraint. 

\section{Augmented Lagrangian Method}\label{sec:ALM}
The AL method is used to solve optimization problems with equality and/or inequality constraints \cite{Nocedal1999}:
\begin{equation} \label{eq:constr_min}
\begin{split}
\min_x f(x), \quad \mathrm{s.t.} \quad h(x)=0, \quad g(x) \leq 0.
\end{split}
\end{equation}
More specifically we use as inequality constraint: $l \leq x \leq u$, which bounds the factor matrix elements to the interval $[l,u]$. 
The AL method is a combination of the Lagrange multipliers method and the QP method. The Augmented Lagrangian (AL) objective function is 
\begin{equation*} 
\begin{split}
\min_{x,y_1,y_2,z}  \underbrace{f(x) + \left\langle y_1,h(x) \right\rangle + \langle y_2, x - z \rangle + \frac{\beta}{2} \left( \lVert h(x) \rVert^2 + \lVert x- z\rVert^2 \right)}_{=:\mathcal{L}_A \left( x,y_1, y_2,z;\beta \right)}, \quad \mathrm{s.t.} \quad l \leq z \leq u,
\end{split}
\end{equation*}
where $y_1$ and $y_2$ are vectors of Lagrange multipliers, $\beta$ is the regularization parameter, and $z$ is a vector of slack variables that enables us to rewrite the inequality constraint as an equality constraint. The objective $\mathcal{L}_A \left( x,y_1, y_2,z;\beta \right)$ can be rewritten as a least-squares function:
\begin{equation*}
\begin{split}
\mathcal{L}_A \left( x,y_1, y_2,z;\beta \right) = f(x) + \frac{\beta}{2} \left\lVert h(x) +  \frac{y_1}{\beta} \right\rVert^2 + \frac{\beta}{2} \left\lVert x -z +  \frac{y_2}{\beta} \right\rVert^2 - \frac{1}{2 \beta} \lVert y \rVert^2
\end{split}
\end{equation*}
s.t. $l \leq z \leq u$, where $y:=[y_1;y_2]$. In every iteration of the AL method we search for the optimal parameters $x$ and $z$ given $y$ and $\beta$:
\begin{equation} \label{eq:min_ALM_objective_LSQ}
\begin{split}
\min_{\substack{x \\ l \leq z \leq u}} \mathcal{L}_A \left( x,z ; y,\beta \right) &= \min_{x} \min_{l \leq z \leq u}  \left( \frac{2}{\beta} f(x) +  \left\lVert h(x) +  \frac{y_1}{\beta} \right\rVert^2 +  \left\lVert x -z +  \frac{y_2}{\beta} \right\rVert^2 \right) \\
&= \min_{x} \left(  \frac{2}{\beta} f(x) +  \left\lVert h(x) +  \frac{y_1}{\beta} \right\rVert^2 + \min_{l \leq z \leq u}  \left\lVert x -z +  \frac{y_2}{\beta} \right\rVert^2 \right).
\end{split}
\end{equation}
The optimal parameters $z^*$ in function of $x$ can easily be found:
\begin{equation} \label{eq:z_optimal}
z^* = \begin{cases}
l& \text{ when } x + \frac{y_2}{\beta} < l, \\
x +  \frac{y_2}{\beta} & \text{ when } l \leq x + \frac{y_2}{\beta}  \leq u, \\
u & \text{ when } u < x + \frac{y_2}{\beta},  \\
\end{cases}
\end{equation}
and thus
\begin{equation*} 
\min_{l \leq z \leq u} \left\lVert x - z + \frac{y_2}{\beta} \right\rVert^2 = \begin{cases}
\left\lVert x - l + \frac{y_2}{\beta} \right\rVert^2 & \text{ when } x + \frac{y_2}{\beta} < l, \\
0 & \text{ when } l \leq x + \frac{y_2}{\beta}  \leq u, \\
\left\lVert x - u+ \frac{y_2}{\beta} \right\rVert^2 & \text{ when } u < x + \frac{y_2}{\beta},  \\
\end{cases}
\end{equation*}
which can be written more compactly as $\left\lVert \left[ l - x - \frac{y_2}{\beta} \right]_+ + \left[ x+ \frac{y_2}{\beta} -u \right]_+ \right\rVert^2$, where $\left[ w \right]_+ = \max\lbrace w,0 \rbrace$. This function is taken element wise. The update of the Lagrange multipliers is \cite{Nocedal1999}: $y_{k+1} = y_{k} + \beta \begin{bmatrix}
h(x_k) &~ x_k - z_k \end{bmatrix}^\top$.
Using \eqref{eq:z_optimal} we can rewrite the update of $y_{k+1,2}$ as:
\begin{equation*} 
\begin{split}
y_{k+1,2} =  \beta \left( \frac{y_{k,2}}{\beta} + x_k - z_k \right)  &= \beta \left( \left[\frac{y_{k,2}}{\beta} + x_k - u) \right]_+ - \left[ l - \frac{y_{k,2}}{\beta} - x_k \right]_+ \right)\\
&=  \left[ y_{k,2} + \beta ( x_k - u) )\right]_+ -  \left[ \beta ( l - x_k) - y_{k,2}  \right]_+.
\end{split}
\end{equation*}
For the minimization to $x$ we use the Levenberg-Marquardt method and rewrite \eqref{eq:min_ALM_objective_LSQ} as 
\begin{equation} \label{eq:AL_noz}
\begin{split}
\min_{x} \left( f(x) + \frac{\beta}{2} \left\lVert h(x) +  \frac{y_1}{\beta} \right\rVert^2 + \frac{\beta}{2} \left\lVert \left[ l - x - \frac{y_2}{\beta} \right]_+ +  \left[ x+ \frac{y_2}{\beta} -u \right]_+ \right\rVert^2 \right). \\
\end{split}
\end{equation}
This objective function can be considered as the objective function of the QP method for the minimization of $f(x)$ with constraint
\begin{equation*} 
\begin{split}
h_\beta(x;y) &:=  \begin{bmatrix}
h(x) +  \frac{y_1}{\beta} \\
\left[ l - x - \frac{y_2}{\beta} \right]_+ +
\left[ x+ \frac{y_2}{\beta} -u \right]_+ 
\end{bmatrix},
\end{split}
\end{equation*}
and regularization parameter $\frac{\beta}{2}$. We can rewrite the objective function in \eqref{eq:AL_noz} more compactly as
\begin{equation} \label{eq:obj_LM_ALM}
\begin{split}
&\min_{x}   \frac{\beta}{2} \left\lVert  \underbrace{\begin{bmatrix}
\frac{1}{\sqrt{\beta}}\left( F(x) - \mathrm{vec} \left( T_{mpn} \right)\right) \\
h_\beta(x;y)
\end{bmatrix}}_{=:F_{\beta}(x;y)} \right\rVert^2.
\end{split}
\end{equation}

The pseudocode of the AL method is shown in \Cref{alg:ALM} \cite[Algorithm 17.4]{Nocedal1999}. 
After the initialization, the algorithm searches an approximate minimizer $x_{k}$ of \eqref{eq:obj_LM_ALM} on line 3 using the LM method described in \Cref{sec:LM}. We run the LM method for a fixed number of iterations $k_{\max}$, e.g., 20-50 or until the norm of the gradient is smaller than $\omega_{k-1}$, which is updated dynamically. 

The gradient of \eqref{eq:obj_LM_ALM} is
\begin{equation*}
\nabla_x F_{\beta} \left( x;y \right)= \frac{1}{{\beta}} \nabla f(x) +  J_h(x)^\top \left(h(x) + \frac{y_1}{\beta} \right) - \left[ l - x - \frac{y_2}{\beta} \right]_+ + \left[ x+ \frac{y_2}{\beta} -u \right]_+,
\end{equation*}
where $J_h(x)$ is the Jacobian matrix of $h(x)$. For the LM method, we also need the Gauss-Newton Hessian approximation, which is 
\begin{equation*}
\begin{split}
\tilde{H}_{\beta}(x;y) 
&=  \frac{1}{{\beta}} J_{F} (x) ^\top J_{F}(x) + J_{h}(x)^\top J_{h}(x) + I_g, \\
\end{split}
\end{equation*} 
where 
\begin{equation*}
I_g(i,j) := \begin{cases} 
1 & \text{if } i = j \text{ and } x(i)+ \frac{y_2(i)}{\beta} < l, \\
1 & \text{else if } i = j \text{ and } u < x(i)+ \frac{y_2(i)}{\beta}, \\
0 & \text{else}. \\ \end{cases}
\end{equation*} 

After computing $x_k$, the update of the other parameters depends on the feasibility of $x_k$. 
A good value of the auxiliary parameters together with a convergence analysis is given in \cite{Conn1991}. More specifically they suggest $\beta_{0}:=10$, $\alpha_\omega := \beta_\omega := \bar{\eta} := \bar{\omega} := 1$, $\alpha_\eta := \bar{\gamma} := 0.1$,  $\beta_{\eta} := 0.9$, and $\tau := 0.01$ for well-scaled problems. We suggest to higher the parameter $\tau$ to, e.g., 0.5 because otherwise the method converges easily to local minima. 

\begin{algorithm}
\caption{AL Method for the Minimization of \eqref{eq:constr_min}}\label{alg:ALM}
\begin{algorithmic}[1]
\Require $x_0,y_0, \beta_0 ~(\geq 1), \alpha_\omega, \beta_\omega,  \bar{\eta} ~(\leq 1), \bar{\omega}  ~(\leq 1), \alpha_{\eta} ~(< \min \left(1,\alpha_{\omega}\right)), \bar{\gamma} ~(< 1), \beta_{\eta} ~(< \min \left(1,\beta_{\omega}\right)), \tau ~(> 1)$, $\omega^*, \eta^* (< 1)$
\State Initialize tolerances:
$\eta_0 \leftarrow \bar{\eta} \left( \min \left(\beta_0^{-1},\bar{\gamma} \right) \right)^{\alpha_{\eta}},~   \omega_0 \leftarrow \bar{\omega} \min(\beta_0^{-1}, \bar{\gamma})^{\alpha_{\omega}}$
\For {$k=1 \dots k_{\max}$} 
\State Find $x_{k}$ such that $\left\lVert  \nabla_x F_{\beta_{k-1}} (x_{k};y_{k-1}) \right\rVert < \omega_{k-1}$, using \Cref{alg:LM}.
\If {$\lVert h_{\beta_{k-1}} (x_k;y_{k-1}) \rVert < \eta_{k-1}$}
\If {$ \lVert h_{\beta_{k-1}} (x_k;y_{k-1})  \rVert < \eta^* \And \left\lVert  \nabla_x F_{\beta_{k-1}} (x_{k};y_{k-1}) \right\rVert < \omega^*$}
\State \textbf{Output:} $x^* \leftarrow x_k$
\EndIf
\State $y_{k,1} \leftarrow y_{k-1,1} + \beta_{k-1} h(x_k)$
\State $y_{k,2} \leftarrow \left[ y_{k-1,2} + \beta_{k-1} ( x_k - u) )\right]_+ -  \left[ \beta_{k-1} ( l - x_k) - y_{k-1,2}  \right]_+$
\State $\beta_{k} \leftarrow \beta_{k-1}$
\State $\eta_{k} \leftarrow \eta_{k-1} \beta_{k}^{-\beta_{\eta}}  $
\State $\omega_{k} \leftarrow \omega_{k-1} \beta_{k}^{-\beta_{\omega}}  $
\Else
\State $y_{k} \leftarrow y_{k-1}$
\State $\beta_{k} \leftarrow  \tau \beta_{k-1}$
\State $\eta_{k} \leftarrow \bar{\eta} \bar{\gamma}^{\alpha_{\eta}} \beta_{k}^{-\alpha_{\eta}}  $
\State $\omega_{k} \leftarrow \bar{\omega} \bar{\gamma}^{\alpha_{\omega}}\beta_{k} ^{-\alpha_{\omega}} $
\EndIf
\EndFor
\Ensure $x_k$
\end{algorithmic}
\end{algorithm}

\subsection{Equality Constraints} \label{sec:constraints}

To obtain new discrete PDs, we start from a numerical solution of \eqref{eq: unconstr_LSQ_min}, and use the equality constraint
$h_{\mathrm{discr}} (x):= x\cdot (x-1) \cdot (x+1)$,
where `$\cdot$' indicates elementwise multiplication. The convergence is in general very fast when the PD is discretizable, as discussed in \Cref{sec:discret_inv_trans}.

Another possibility, for example for non-discretizable PDs, is to start from a solution of \eqref{eq: unconstr_LSQ_min} and use the equality constraint
$h_I(x):= x(I)$,
where $I$ is the set of indices $i\in \lbrace 1 \dots (mp+pn+mn)R\rbrace$ for which $\lvert x^*(i) \rvert$ is smaller than a certain tolerance, e.g., $10^{-2}$. This procedure can be applied recursively if necessary. The PDs that are obtained are sparse but in general not discrete. However, from these sparse numerical PDs parameter families can be discovered. These parametrizations and the procedure to obtain them are discussed in more detail in \Cref{sec:results}. 

To investigate additional invariances or parametrizations we also use the vectorization of a discrete and sparse PD $x^*$ and add a random perturbation, i.e., $x_0 := x^* + 10^{-2} \mlin{randn}((mn+mp+pn)R,1)$, where $\mlin{randn}$ is a build-in {\Matlab} function to generate normally distributed pseudorandom numbers. We use such a starting point in combination with the equality constraint
$h_I(x)$ defined above,
but where $I$ is the set of indices for which $x^*(i)=0$. This equality constraint ensures that no additional nonzeros are created in the output of the AL method. Remark that the PQR-transformation from \Cref{sec:inv_transf} in general creates additional nonzeros and thus other invariances can separately be investigated with this constraint.

\subsection{Starting Points} \label{sec:starting_points}

If the starting points are generated randomly, e.g, $x_0(i) \sim \mathcal{N}(0,1)$, for $i=1 \dots (mp+np+mn)R$, we observed that $f(x_0)$ is, in general, large and consequently the LM method gets easily stuck in regions of very slow convergence around local minima of \eqref{eq: unconstr_LSQ_min}, which typically occur at objective values which are a multiple of 0.5. 
These regions are called swamps and an example convergence with different swamps is shown in \Cref{fig:swamps} for $T_{333}$.

\begin{figure}[ht]
\centering
\includegraphics[width= 0.5\linewidth]{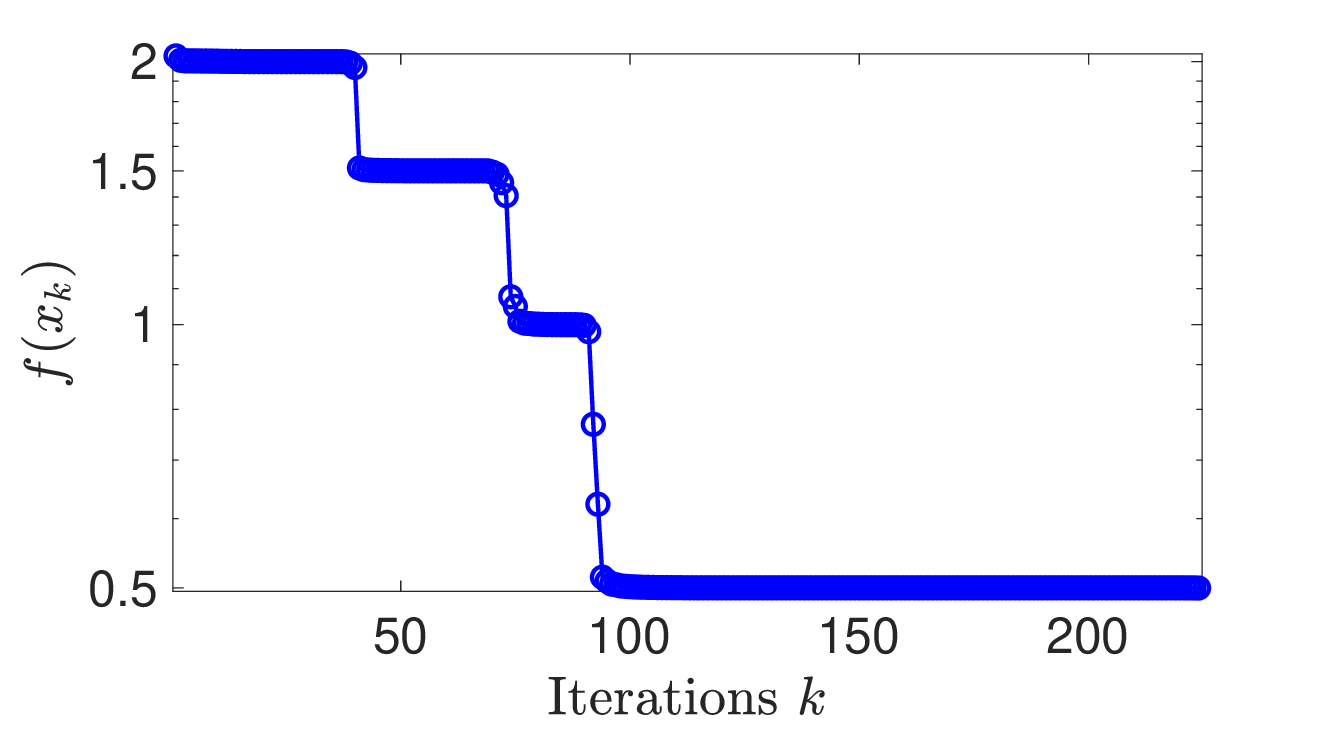}
\caption{Illustration of the swamps around stationary points of \eqref{eq: unconstr_LSQ_min} for $T_{333}$ and the standard LM method described in \Cref{sec:LM}.}
\label{fig:swamps}
\end{figure}

Because of the structure of $T_{mpn}$, containing $mpn$ ones, we know that for $x_0:=0$, $f(x_0)$ is $0.5(mpn)^2$. However, the gradient and Hessian at $x_0$ are zero as well, which means that $x_0$ is an inflection point and thus we need to add a small perturbation. 
We found that a good size of the perturbation is, e.g., $10^{-2}$ or $x_0(i) \sim \mathcal{N}\left(0,10^{-4}\right)$. In practice, we generate such starting points with $\mlin{randn}$ in \Matlab. Because the perturbation is small, the objective only increases slightly. Another possibility is to use uniformly distributed elements in $[0,1]$ using $\mlin{rand}$.
An example convergence of the AL method with $u:=-l:=1$ for 20 random starting points generated with \mlin{randn} of size $10^{-1}$ for $T_{333}$ and rank 23 is shown in \Cref{fig:randn_10-1_rng1-20_R23_ALM_ul1}. As can be seen, the method converges to a numerical solution of \eqref{eq: unconstr_LSQ_min} up to machine precision for 8 out of the 20 starting points. As an illustration, 5 different values for the rank of the Jacobian matrix are found among these numerical solutions (538, 539, 543-545), and at least one of these PDs is not discretizable as it does not satisfy one of the necessary conditions.
If we lower the rank to 22, the convergence for 100 random starting points is shown in \Cref{fig:10-1_rand_ALM_ul1_R22_rng1-100_cost}. As can be seen, the most accurate value of the cost function that we are able to obtain is approximately $10^{-2}$ for $u:=-l:=1$. We suspect that these approximate solutions converge to BRPDs.

\begin{figure}[htbp]
\captionsetup[subfigure]{justification=centering}
\begin{subfigure}{0.49\textwidth}
\includegraphics[width= \textwidth]{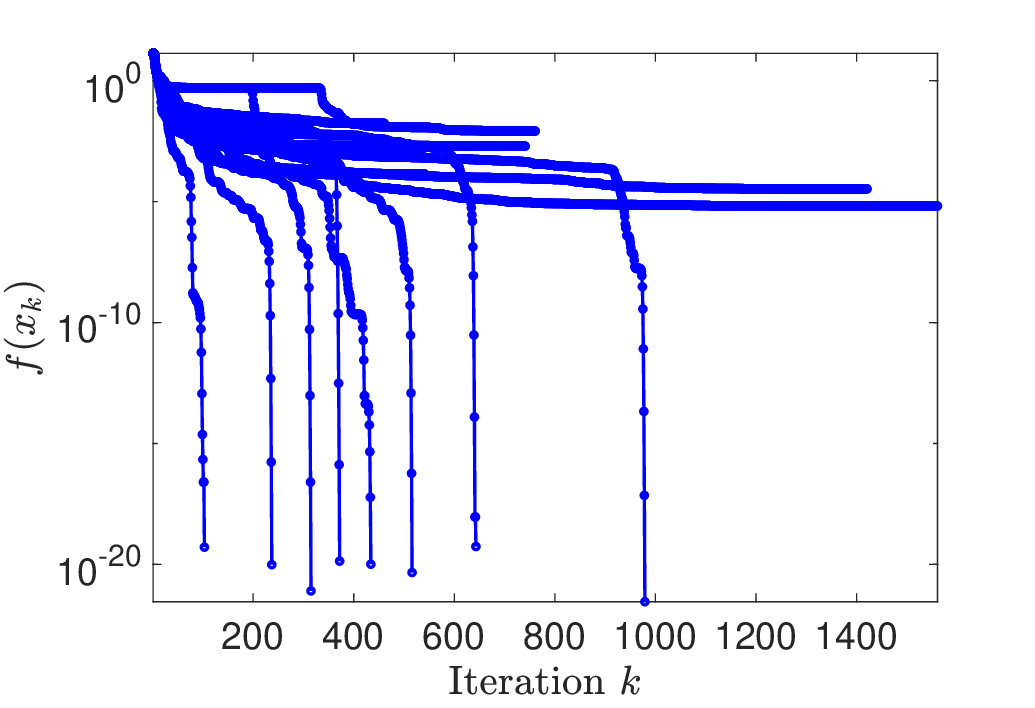}
\caption{20 random starting points for $R:=23$.}
\label{fig:randn_10-1_rng1-20_R23_ALM_ul1}
\end{subfigure}
\begin{subfigure}{0.49\textwidth}
\centering
\includegraphics[width= \textwidth]{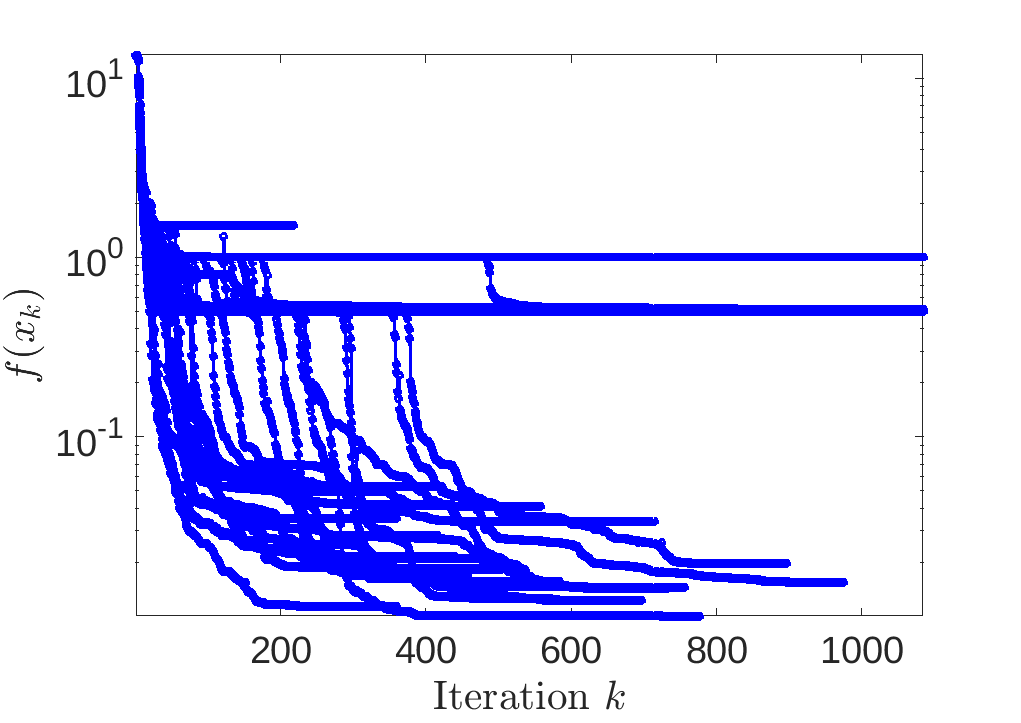}
\caption{100 random starting points for $R:=22$.}
\label{fig:10-1_rand_ALM_ul1_R22_rng1-100_cost}
\end{subfigure}
\caption{Example convergence of the AL method with $u:=-l:=1$ for different random starting points generated using $\mlin{randn}$ of size $10^{-1}$ in {\Matlab} for $T_{333}$.}
\end{figure}

The initial cost function can further be lowered by choosing the points such that each rank one tensor generates one of the $mpn$ ones in the tensor. Such starting points have a cost function of $0.5(mpn-R)^2$. These discrete starting points are also stationary points of \eqref{eq: unconstr_LSQ_min} in general, and thus we add again a perturbation of size, e.g., $10^{-2}$, such that $x_0:= x_{0,\mathrm{discr}} + 10^{-2} \mlin{randn}((mp+np+mn)R,1)$. Note that there are $\frac{mpn!}{R!(mpn-R)!}$ such combinations of starting points. For $T_{333}$, rank 23, there are $17550$ combinations. The convergence of these starting points can be very fast but equally very slow. It is thus important to try sufficiently many combinations.

We noticed that for larger values of $m$, $p$, and $n$, it can be a good idea to start from a known discrete PD, set all elements in the factor matrices that do not generate one of the $mpn$ ones in the tensor to zero, add a small perturbation, and use this as starting point to obtain new (discrete) PDs. We show in \Cref{fig:recursive_sol_Strassen_m4_R49_triv1s+10-1_randn_ALM_ul1_rng5-10_cost} the convergence for 6 random starting points for $T_{444}$, rank 49, based on these 'trivial' ones in $\mathrm{PD}_{\mathrm{Strassen}}^\mathrm{rec}$. All starting points converge quickly to a numerical exact PD and two different values of $\rank(J(x^*))$ are obtained: 2155 and 2182. As a reference, the rank of $J(x^*)$ at $\mathrm{PD}_{\mathrm{Strassen}}^\mathrm{rec}$ is 2101. All PDs satisfy the necessary conditions to be discretizable discussed in \Cref{sec:discret_inv_trans}. 

\begin{figure}[ht]
\captionsetup[subfigure]{justification=centering}
\begin{subfigure}{0.49\textwidth}
\includegraphics[width= \textwidth]{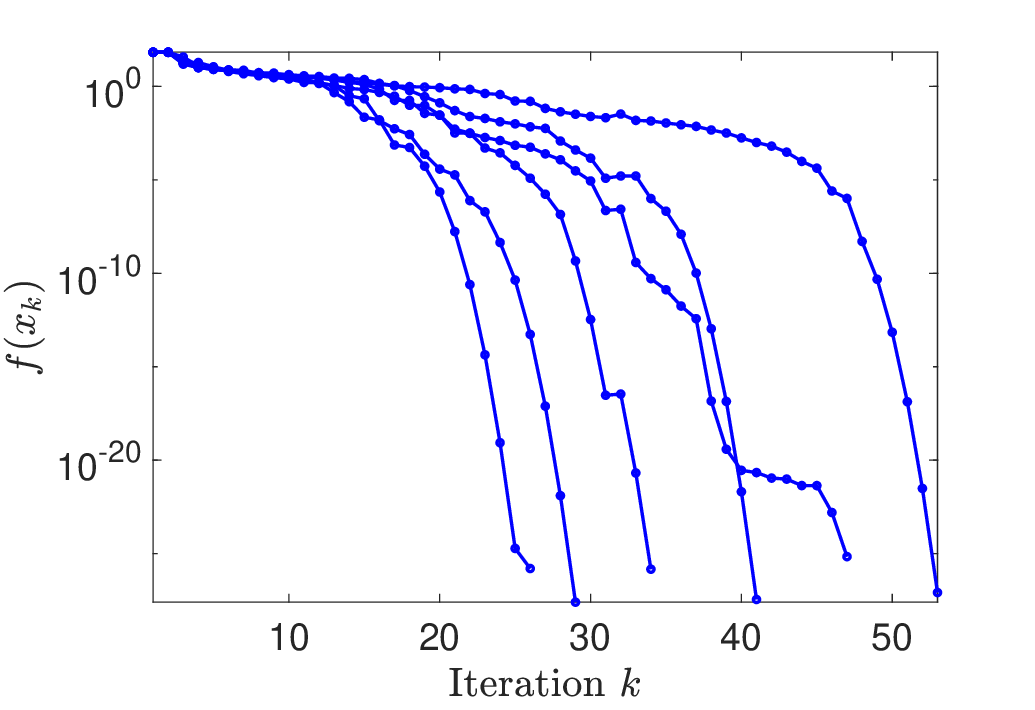}
\caption{Six starting points obtained using the 'trivial' $mpn$ ones in $\mathrm{PD}_{\mathrm{Strassen}}^{\mathrm{rec}}$ and with a small random perturbation of size $10^{-1}$.}
\label{fig:recursive_sol_Strassen_m4_R49_triv1s+10-1_randn_ALM_ul1_rng5-10_cost}
\end{subfigure}
\begin{subfigure}{0.49\textwidth}
\centering
\includegraphics[width= \textwidth]{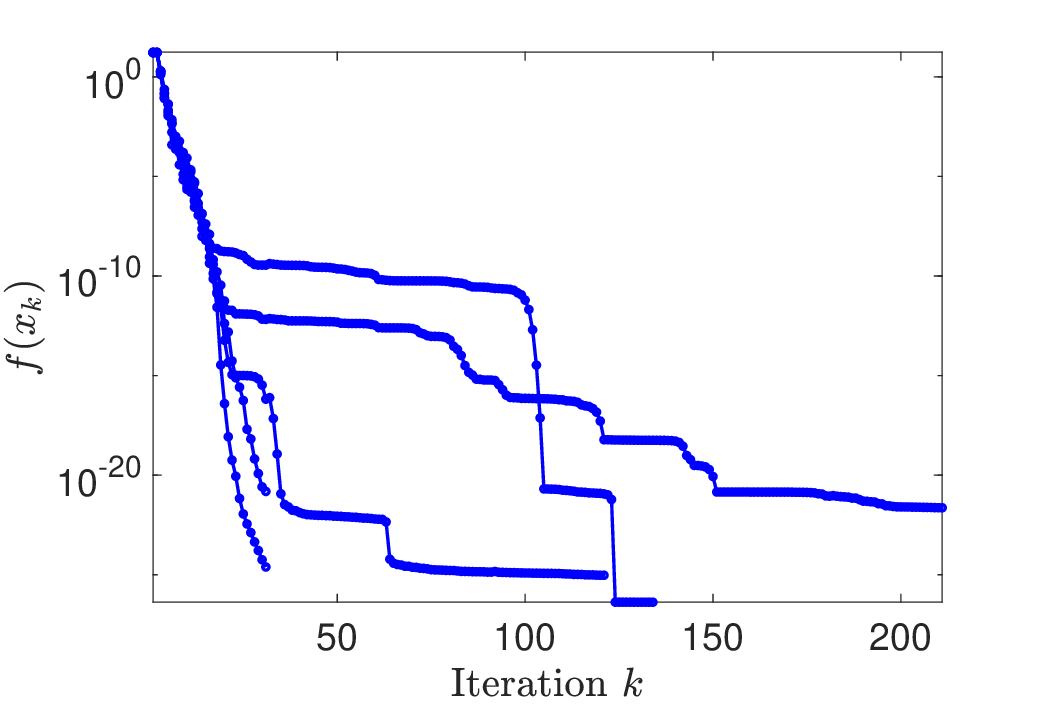}
\caption{Five random starting points obtained by adding a small random perturbation of size $10^{-1}$ to $\mathrm{PD}_{\mathrm{Strassen}}^{\mathrm{rec}}$.}
\label{fig:recursive_sol_Strassen_m4_R49_10-1_randn_ALM_ul1_rng1-5_cost}
\end{subfigure}
\caption{Convergence of the AL method for $u:=-l:=1$ for different starting points for $T_{444}$, rank 49.}
\end{figure}

Lastly, when we add a perturbation, e.g., of size $10^{-1}$, to a known (discrete) vectorized PD $x^*$:
\begin{equation}
x_0 := x^* + 10^{-1} \mlin{randn}((mn+mp+pn)R,1),
\end{equation}
the convergence is in general very fast. An example convergence for 5 starting points for $T_{444}$, rank 49, based on $\mathrm{PD}_{\mathrm{Strassen}}^{\mathrm{rec}}$, is shown in \Cref{fig:recursive_sol_Strassen_m4_R49_10-1_randn_ALM_ul1_rng1-5_cost}. Again, all starting points converge to a numerical solution up to machine precision. Because we start close to a known PD, the rank of the Jacobian stays in general the same. All PDs again satisfy the necessary conditions to be discretizable. Remark that starting from a perturbation of $x^*$ can be slower compared to when we only start from the $mpn$ ones because the norm of the starting point is higher in this case and the bound constraint in the AL method restricts the norm. This can be resolved by making the upper and lower bound higher in absolute value, e.g., $u:=-l:=2$. When adding a constraint on the number of nonzeros, these starting points can be used to discover parameter families of PDs.

\section{Results}\label{sec:results}
In this section, we give the results that we obtained with the AL method described in \Cref{sec:ALM}. In the first subsection, we give an example of an additional invariance that we noticed frequently among PDs of $T_{mpn}$ and furthermore we describe our procedure to obtain parameter families. Then, in \Cref{sec:param_and_stab_improv}, we give the different parametrizations that we have obtained together with the optimal parameters for which an improvement of the stability is attained. Finally, in \Cref{sec:approx_decomp}, we give the cost function and bounds $u$ and $l$ of the most accurate approximate PDs that we obtained for ranks lower than $\tilde{R}$.

\subsection{Additional Invariances} \label{sec:add_inv}
When investigating existing PDs, we noticed that some subsets of the rank-1 tensors in \eqref{eq:CPD} have additional invariances. These additional invariances can be grouped together into a parametrization or a parameter family of PDs of $T_{mpn}$. For example, we noticed that many (discrete) PDs satisfy $u_i = u_j = u$, for some $i\neq j$:
\begin{equation*}
T_{ij} := u \otimes v_i \otimes w_i + u \otimes v_j \otimes w_j.
\end{equation*} 
The Jacobian $J(x)$ at this PD has a rank deficiency of 2 compared to generic rank-2 PDs and indeed, a two-parameter family of PDs of $T_{ij}$ exists:
\begin{equation*}
T_{ij} = \frac{u}{1-ab} \otimes \left( v_i + av_j \right) \otimes \left( w_i - b w_j \right) + \frac{u}{1-ab} \otimes \left( v_j + b v_i \right) \otimes \left( w_j - a w_i \right),
\end{equation*}
for all $a,b \in \mathbb{R}$, as long as $ab \neq 1$. Remark that the scaling invariance is not included in this parametrization. Furthermore, in general it holds that
\begin{equation*}
\begin{split}
&\mathrm{sort} \left( \begin{bmatrix} \mathrm{trace} \left( U V_i W_i\right) & \mathrm{trace} \left( U V_j W_j \right) \end{bmatrix} \right) \neq~ \\
 &\mathrm{sort} \left( \begin{bmatrix} \mathrm{trace} \Big( \frac{1}{1-ab} U \left(V_i + a V_j \right) \left(W_i - b W_j \right)\Big) & \mathrm{trace} \left( \frac{1}{1-ab} U \left(V_j + b V_i \right) \left(W_j - a W_i \right)  \right)  \end{bmatrix} \right)
\end{split}
\end{equation*}
where `$\mathrm{sort}$' represents a sorting operation, and the non-equality is applied element wise. Consequently, this transformation or parametrization is not trace-invariant, which means that two PDs that are not inv-equivalent can still be equivalent by this (or another) continuous transformation.
We encountered many similar examples of additional invariances or parametrizations. The advantage of these parametrizations is that they can be used to find the most stable PD within the family of PDs.

As explained in \Cref{sec:Jacobian}, the rank deficiency of the Jacobian matrix can be used as an upper bound on the number of parameters that can be introduced. In \Cref{tab:svdJ_N}, the different new ranks we found for $\mathrm{PD}_{47} \left( T_{345} \right)$ and $\mathrm{PD}_{49} \left( T_{444} \right)$ are given (underlined) and compared with the ranks at the PDs that were recently found by DeepMind (bold) \cite{DeepMind_2022}. As a reference, we add the rank at $\mathrm{PD}_{\mathrm{Strassen}}^{\mathrm{rec}}$, which is 2101. Similarly, the rank at all $\mathrm{PD}_7 \left( T_{222} \right)$s is 61. 

The rank of $J(x)$ is however not useful if we want to investigate the additional invariances and not the scaling- and PQR-invariance from \Cref{sec:inv_transf}. The PQR-transformation creates in general additional nonzeros in the factor matrices. That is why we use the constraint $h_I(x)$, where $x^*(I)=0$, as explained in \Cref{sec:constraints}. An upper bound on the number of parameters that can be introduced in the sparse $\mathrm{PD}_R$ $x^*$ can be obtained by:
\begin{equation*}
\min \left( \mathrm{size} \left( \begin{bmatrix}
J(x^*) \\
J_{h_I}(x^*)
\end{bmatrix} \right) \right) - \mathrm{rank} \left( \begin{bmatrix}
J(x^*) \\
J_{h_I}(x^*)
\end{bmatrix} \right) - 2R,
\end{equation*}
where $J_{h_I}(x^*)$ is the Jacobian matrix of $h_I(x)$, evaluated at $x^*$, and $\mathrm{size}(\cdot)$ returns a tuple containing the dimensions of a matrix. 

When looking at the elements of a sparse numerical PD, a certain pattern can be recognised. For example some values occur more than once. We introduce the parameters manually by first choosing as parameters the values that appear most frequently in the PD. Note that a PD, see \eqref{eq:CPD}, is a set of third order polynomial equations in the elements of the factor matrices. Because the PDs we use are sparse, these equations only have few terms, preferably one or two. Using these simple equations, the other numerical values can be written in function of the chosen parameters. This can be a process of trial and error but is in general relatively easy.

To improve the stability, we first check if some of the parameters can be chosen such that the PD is more sparse. When this is the case, the prefactor $q$ is in general lowered. Afterwards, we import the obtained parametrizations in {\Matlab} and use the build-in nonlinear solver $\mlin{fmincon}$ to obtain an approximate minimizer for the stability parameter $e$ \eqref{eq:e_q}. Finally we round the obtained numerical parameters to the nearest power of two to obtain a practical FMM algorithm. 

\begin{table}[htbp]
\centering
\caption{Rank of $J(x^*)$ for different values of $m$, $p$, and $n$. The bold ranks indicate the ranks among PDs found by DeepMind \cite{DeepMind_2022}. The underlined ranks indicate the ranks of new PDs from this paper.}
\label{tab:svdJ_N}
\begin{tabular}{c c c c | c c } 
\boldmath $m$ & \boldmath $p$ & \boldmath $n$ & \boldmath $\tilde{R}$ & \boldmath $\mathrm{size}(J)$ & \boldmath $\mathrm{rank}\left( J(x^*) \right)$ \\
\midrule
2 & 2 & 2 & 7 & $64 \times 84$ & $61$ \\
3 & 4 & 5 & 47 & $ 3600 \times 2209$ & \mbox{ \boldmath $2051$}, $\underline{2053}$\\
\multirow{2}{*}{4} & \multirow{2}{*}{4} & \multirow{2}{*}{4} & \multirow{2}{*}{49} & \multirow{2}{*}{$4096 \times 2352$} & $2101$, \mbox{\boldmath $2144,2145,2149-2155,2158,2161,2169,2170$}, \\
& & & &  & \mbox{ \boldmath $2172-2176$}, $\underline{2182,2184,2188,2195}$, \mbox{\boldmath $2196-2201$} \\
\bottomrule
\end{tabular}
\end{table}

\subsection{Parametrizations and Stability Improvements}\label{sec:param_and_stab_improv}

The stability improvements and correspondingly the different parametrizations are summarized in \Cref{tab:results_param} for different combinations of $m,p$, and $n$. We always choose the lowest rank $\tilde{R}$ known in the literature for a certain combination of $m$, $p$, and $n$, as we were not able to find exact PDs of lower rank in our experiments. 

\begin{table}[htbp]
\caption{Overview of the different stability improvements, corresponding parametrizations, and optimal parameters for different values of $m$, $p$, and $n$.}
\label{tab:results_param}
\centering
\resizebox{\textwidth}{!}{\begin{minipage}{\textwidth}
\begin{tabular}{c c c | c c c c c c}
\mbox{\boldmath $m$} & \boldmath $p$ & \boldmath $n$ & \boldmath $\tilde{R}$ & \boldmath $(q,e)_{\text{old}}$ & Ref. & \boldmath $(q,e)_{\text{new}}$ & Param. & Optimal Parameters \\
\midrule
\multirow{2}{*}{3} & \multirow{2}{*}{3} & \multirow{2}{*}{3} & \multirow{2}{*}{23} & \multirow{2}{*}{$(13,31)$} & \multirow{2}{*}{\cite{Smirnov2013}}  & \multirow{2}{*}{$(13,27.25)$} & \multirow{2}{*}{\Cref{tab:9_param_Smirnov}} & $(b,d,g):=\frac{1}{4},(k,p):=\frac{1}{2}$, \\
& & & & & & & & $(a,h):=1,(c,f):=2$ \\
5 & 4 & 3 & \multirow{6}{*}{47} & $(21,112)$ & \cite{DeepMind_2022} & $(20,66)$ & \Cref{tab:param_R47_U_1,tab:param_R47_U_2,tab:param_R47_V,tab:param_R47_W} & $a:=\frac{1}{4},b:=2,c:=\frac{1}{2}$ \\
4 & 5 & 3 & & $(26,99)$ & \cite{DeepMind_2022} & $(25,88)$ & \Cref{tab:param_R47_U_1,tab:param_R47_U_2,tab:param_R47_V,tab:param_R47_W} & $a:=1,b:=2,c:=\frac{1}{2}$ \\
5 & 3 & 4 & & $(19,83)$ & \cite{DeepMind_2022} & $(18,75)$ & \Cref{tab:param_R47_U_1,tab:param_R47_U_2,tab:param_R47_V,tab:param_R47_W} & $(a,b,c):=1$ \\
3 & 4 & 5 & & $(21,112)$ & \cite{DeepMind_2022} & $(20,66)$ & \Cref{tab:param_R47_U_1,tab:param_R47_U_2,tab:param_R47_V,tab:param_R47_W} & $a:=1,b:=2,c:=0.5$ \\
4 & 3 & 5 & & $(19,83)$ & \cite{DeepMind_2022} & $(18,75)$ & \Cref{tab:param_R47_U_1,tab:param_R47_U_2,tab:param_R47_V,tab:param_R47_W} & $(a,b,c):=1$ \\
3 & 5 & 4 & & $(26,99)$ & \cite{DeepMind_2022} & $(25,86.5)$ & \Cref{tab:param_R47_U_1,tab:param_R47_U_2,tab:param_R47_V,tab:param_R47_W} & $a:=1,b:=4,c:=\frac{1}{4}$ \\
\multirow{2}{*}{4} & \multirow{2}{*}{4} & \multirow{2}{*}{4} & \multirow{2}{*}{49} & \multirow{2}{*}{$(23,136)$} & \multirow{2}{*}{\cite{DeepMind_2022}}  & \multirow{2}{*}{$(22,99)$} & \multirow{2}{*}{\Cref{tab:param_R49_U,tab:param_R49_V,tab:param_R49_W}} & $(b,f,g,h,p,t):=\frac{1}{2}$, \\
& & & & & & & & $(c,d,n,s,w):=1,a=2,v:=4$ \\
\end{tabular}
\end{minipage}}
\end{table}

\paragraph*{Rank 23 of \boldmath $T_{333}$}
A two-parameter family of a non-discretizable PD is given in \Cref{tab:non-discr_R23}. The parameters $a$ and $b$ can be chosen such that all elements are rational numbers but not discrete numbers.
The PDs in this parametrization contain more nonzeros than the most sparse PD known \cite{Smirnov2013}. The rank of $J(x^*)$  at a generic PD (randomly generated parameters) in this family is 541. Remark that when $a$ is close to $-1$ or $0$, some elements in the PD go to infinity. We obtained this parametrization by performing sufficiently many experiments for random starting points close to zero, i.e., $x_0:= 10^{-2} \mlin{randn} \left( (mp+pn+mn)R,1 \right)$, and $u:=-l:=1$. When we find a numerical solution $x^*$ up to machine precision, we add the equality constraint $h_I(x)$, based on the element in $x^*$ smaller than a certain tolerance, e.g., $10^{-2}$. This process can be applied recursively if necessary. When we obtain a sparse numerical solution, we can obtain the parametrization manually as described at the end of \Cref{sec:add_inv}.

In \Cref{tab:9_param_Smirnov}, a nine-parameter family of PDs is shown, which is obtained starting from:
\begin{equation*}
x_0 := x^* + 10^{-1} \mlin{randn} \left( (mp+pn+mn)R,1 \right),
\end{equation*}
where $x^*$ is the vectorization of the discrete PD found by Smirnov \cite{Smirnov2013} of rank 23. We use the equality constraint $h_I(x)$, based on the zeros in $x^*$, and $u:=-l:=2$ to obtain a numerical sparse PD in this family and afterwards obtain the parametrization manually again as described at the end of \Cref{sec:add_inv}. Using this parametrization, we reduce the stability parameter from $e=31$ to $e=27.25$ for $(b,d,g):=\frac{1}{4},(k,p):=\frac{1}{2},(a,h):=1$, and $(c,f):=2$. The prefactor remains the same and equal to $13$.  
The rank of $J(x)$ is 534 at this PD. We noticed that in general, the sparser the PD, the lower the rank of $J(x)$.

\begin{sidewaystable}
\caption{Non-discretizable two-parameter family of $\mathrm{PD}_{23}(T_{333})$s with parameters $a$ and $b$, where $v = a^3+1$, $w = 1-a$, and $z = \frac{a^4w-bv^2}{va^3}$.}
\label{tab:non-discr_R23}
\resizebox{\textwidth}{!}{\begin{minipage}{\textwidth}
\begin{tabular}{c c |c c c c c c c c c c c c c c c c c c c c c c c }
 & & 1 & 2 & 3 & 4 & 5 & 6 & 7 & 8 & 9 & 10 & 11 & 12 & 13 & 14 & 15 & 16 & 17 & 18 & 19 & 20 & 21 & 22 & 23 \\
\cmidrule{1-25}
\multirow{9}{*}{\begin{sideways}$U$ \end{sideways}} & 1&$0$& $a^2$ & $1$ & $-a^2$ &$0$&$0$&$0$& $-a^2$ &$0$& $a^2$ &$0$&$0$&$0$&$0$&$0$&$0$&$0$&$0$&$0$& $a^2$ &$0$&$0$&$0$\\
& 2 &$0$&$0$&$0$&$0$&$0$&$0$&$0$&$0$&$0$& $1$ & $-1$ &$0$&$0$& $-1$ &$0$&$0$&$0$& $-1$ &$0$&$0$& $\frac{1}{a}$ &$0$&$0$\\
& 3 & $a^2$ & $a^2$ &$0$& $-a^2$ &$0$&$0$& $1$ &$0$&$0$&$0$&$0$&$0$& $a^3$ &$0$&$0$& $1$ & $-1$ & $a^2$ &$0$& $a^2$ &$0$&$0$&$0$\\
& 4 &$0$& $-1$ &$0$&$0$& $1$ &$0$&$0$& $1$ &$0$& $-1$ &$0$&$0$&$0$&$0$&$0$&$0$&$0$&$0$&$0$& $-1$ &$0$&$0$&$0$\\
& 5 &$0$& $a$ &$0$&$0$&$0$& $-a^3$  &$0$&$0$&$0$& $a$ & $-a$ & $1$ &$0$&$0$&$0$&$0$&$0$&$0$&$0$&$0$& $1$ & $\frac{-a}{v}$ &$0$\\
& 6 & $-1$ & $-1$ &$0$& $-a^3$ &$0$&$0$&$0$&$0$&$0$&$0$&$0$&$0$& $-a$ &$0$&$0$& $a$ & $-a$ & $-1$ &$0$& $a^3$ &$0$& $1$ & $1$ \\
& 7 & $\frac{w}{a^2}$	& $\frac{-vw}{a^2}$	& $1$	& $aw$ & $\frac{w}{a^2}$ & $v-a$ &$0$& $\frac{1-av}{a^2}$ & $a-v$ & $\frac{-vw}{a^2}$ & $a$ &$0$&$0$&$0$& $-1$ &$0$& $\frac{1}{a}$ &$0$& $1$ & $\frac{-vw}{a^2}$ &$0$&$0$&$0$\\
& 8 & $1$ & $\frac{w}{a}$	&$0$&$0$&$0$& $2a^2-a$ &$0$&$0$& $-a^2$ &$0$&$0$& $\frac{w}{a^2}$ & $1$ & $-1$ & $a$ &$0$&$0$& $\frac{w}{a}$ &$0$&$0$&$0$& $\frac{-w}{av}$ &$0$\\
& 9 &$0$& $\frac{-vw}{a^2}$ &$0$&$0$&$0$& $vw$ & $1$ &$0$& $-vw$ &$0$&$0$&$0$& $\frac{-vw}{a}$ &$0$& $-v$ &$0$&$0$& $\frac{-vw}{a^2}$ &$0$&$0$&$0$& $\frac{w}{a^2}$ & $\frac{w}{a^2}$ \\ 
\cmidrule{1-25}
\multirow{9}{*}{\begin{sideways}$V$ \end{sideways}} & 1&$\frac{v}{a^2}$ &$0$&$0$& $\frac{1}{a^2}$ &$0$&$0$& $\frac{-w}{a}$ &$0$&$0$&$0$&$0$& 	$0$ & $\frac{w}{a}$ &$0$&$0$& $-a$ & $\frac{w}{a}$ & $1$ &$0$&$0$&$0$& $0$& $-a$ \\ 
& 2 & $ \frac{-v}{a^3}$ & $\frac{-1}{bv}$ & $0$& $1$ & $0$&$0$&	$\frac{w}{a^2}$ &$0$&$0$& 	$0$ &$0$& $0$& $\frac{-w}{a^2}$  &$0$&$0$& $\frac{-w}{a}$ & $ \frac{-w}{a^2} $ &$0$&$0$& $1$ &$0$& $\frac{1}{b}$ & $1$ \\
& 3 &$\frac{v}{aw}$ &$0$&$0$&$0$&$0$&$0$& $-1$ &$0$&$0$&$0$&$0$&$0$& $1$ &$0$&$0$& $a$ & $1$ &$0$&$0$&$0$&$0$&$0$& $a$ \\
& 4 & $a$ & $-1$ & $-1$ & $\frac{a}{v}$ & $-a^2$ & $-a$ &$0$&$0$& $-1$ & $1$ & $1$ &$0$&$0$&$0$& $-aw$ &$0$& $\frac{a^2}{v}$ &  $\frac{bv}{a}$ & $1$ &$0$&$0$& $1$ &$0$\\
& 5 & $-1$ &$0$& $-a^2$ & $\frac{a^3}{v}$  & $-a^4$  & $1$ &$0$& $1$ & $\frac{1}{a}$ & $0$&$0$&$0$&$0$& $0$ & $w$ &$0$& $\frac{a^4}{v}$ &$0$& $\frac{w}{a^2}$	 &$0$&$0$& $a^2$ &$0$\\ 
& 6 &$0$&$0$& $0$ &$0$& $0$ &$0$& $0$ &$0$& 	$1$ &$0$& $0$ & $0$	 &$0$&$0$& $-a^2$ &$0$&$0$&$0$& $-1$ & $0$ & $0$ &$0$&$0$\\
& 7 & $1$ & $\frac{1}{a^2}$ &$0$&$0$& 	$1$ & $\frac{w}{a}$ &$0$&$0$&$0$& $\frac{-1}{a^2}$ &	$\frac{-w}{a^2}$ & $a$ & $\frac{w}{v}$ &$\frac{-w}{a^2}$ & $\frac{w}{a}$ &$0$&$0$& $z$ &$0$&$0$& $1$ & $\frac{-1}{a^2}$ &$0$\\
& 8 & $\frac{-1}{a} $ &$0$&$0$&$0$& $a^2$ & $\frac{-w}{a^2}$ &$0$&$0$&$0$& $0$ & $\frac{w}{a^3}$ & $ -1$ & $\frac{-w}{av}$ & $\frac{w}{a^3}$ & $\frac{-w}{a^2}$ &$0$& $0$ & $\frac{-w}{v}$  &$0$&$0$& $\frac{w}{a^2}$ & $-1$ &$0$\\
& 9 & $\frac{a}{w}$ &$0$&$0$&$0$&$0$& $1$ & $0$ &$0$&$0$ & $0$	& $\frac{-1}{a}$	& $-a$ & $\frac{a}{v}$& $\frac{-1}{a}$ &	$1$ &$0$& $0$ & $\frac{a^2}{v}$	 &$0$& $0$& $-1$ & $0$ &$0$
\\
\cmidrule{1-25}
\multirow{9}{*}{\begin{sideways}$W$ \end{sideways}} & 1&$0$&$0$& $\frac{a^3}{v}$ & $-1$ &$0$&$0$&$0$&$0$&$0$&$0$&$0$&$0$&$0$&$0$&$0$& $1$ & $a$ &$0$&$0$& $-1$ &$0$&$0$&$0$\\
& 2&$0$&$0$& $-1$ &$0$&$0$&$0$&$0$& $1$ &$0$&$0$&$0$&$0$&$0$&$0$&$0$&$0$&$0$& $0$ &$-1$&$0$&$0$&$0$&$0$\\
&3&$0$&$0$& $-a^2$ &$0$& $\frac{1}{a^2}$ &$0$&$0$& $a^2$ &$0$& $-1$ & $-1$ &$0$&$0$&$0$&$0$&$0$&$0$&$0$&$0$&$0$&$-1$&$0$&$0$\\
&4&$0$& $\frac{-bv}{a}$ &$0$&$0$&$0$&$0$&$0$&$0$&$0$& $\frac{-bv}{a}$ &$0$&$0$& $1$ & $\frac{-a^3}{v}$ &$0$& $-a^2$ &$0$& $-1$ &$0$& $\frac{-1}{a}$ &$0$&$0$& $v$ \\ 
&5&$0$&$0$&$0$&$0$&$0$& $\frac{-1}{a^3}$  &$0$&$0$& $\frac{-1}{a}$ &$0$& $-1$ &$0$&$0$& $1$ & $\frac{-w}{a^3}$ &$0$&$0$&$0$& $a^2$ &$0$&$0$&$0$&$0$\\
&6&$0$&$0$&$0$&$0$&$0$&$0$&$0$&$0$&$0$&$0$&$0$& $-a$ &$0$& $a^2$ &$0$&$0$&$0$&$0$&$0$&$0$& $a^2$ &$0$&$0$\\
&7& $0$ & $0$ & $0$ & $0$ & $0$ & $0$ & $-a$ & $0$ & $0$ & $0$ & $0$ & $0$ & $0$ & $0$ & $0$ & $-1$ & $0$ & $0$ & $0$ & $0$ & $0$ & $0$ & $a$ \\
&8& $\frac{1}{v}$ & $0$ & $0$ & $0$ & $0$ & $0$ & $\frac{v}{a^2}$ & $0$ & $\frac{-1}{v}$ & $0$ & $0$ & $0$ & $\frac{-1}{aw}$ & $0$ & $\frac{1}{av}$ & $0$ & $\frac{-1}{a^2}$ & $0$ & $1$ & $0$ & $0$ & $0$ & $0$ \\
&9& $\frac{a^2}{v}$ & $1$ & $0$ & $0$ & $0$ & $\frac{1}{v}$ & $v$ & $0$ & $0$ & $1$ & $1$ & $\frac{-a^2}{v}$ & $\frac{-a}{w}$ & $0$ & $0$ & $0$ & $-1$ & $0$ & $0$ & $\frac{1}{bv}$ & $1$ & $-1$ & $0$ \\ 
\bottomrule
\end{tabular}
\end{minipage}}
\end{sidewaystable}

\begin{sidewaystable}
\caption{Nine-parameter family with parameters $a,b,c,d,f,g,h,k$, and $p$, based on the $\mathrm{PD}_{23}(T_{333})$ of Smirnov \cite{Smirnov2013}.}
\label{tab:9_param_Smirnov}
\resizebox{\textwidth}{!}{\begin{minipage}{\textwidth}
\begin{tabular}{c c |c c c c c c c c c c c c c c c c c c c c c c c }
 & & 1 & 2 & 3 & 4 & 5 & 6 & 7 & 8 & 9 & 10 & 11 & 12 & 13 & 14 & 15 & 16 & 17 & 18 & 19 & 20 & 21 & 22 & 23 \\
\cmidrule{1-25} 
\multirow{9}{*}{\begin{sideways}$U$ \end{sideways}}  & 1 & $1$ & $0$ & $0$ & $0$ & $0$ & $0$ & $1$ & $1$ & $0$ & $0$ & $0$ & $0$ & $0$ & $-ak$ & $\frac{k}{d}$ & $0$ & $0$ & $0$ & $0$ & $0$ & $0$ & $0$ & $0$\\
 & 2 & $0$ & $0$ & $0$ & $1$ & $1$ & $0$ & $0$ & $0$ & $ab$ & $-a$ & $0$ & $0$ & $0$ & $0$ & $0$ & $0$ & $0$ & $0$ & $1$ & $0$ & $1$ & $1$& $0$\\
 & 3 & $0$ & $0$ & $0$ & $c$ & $0$ & $0$ & $0$ & $0$ & $0$ & $0$ & $0$ & $0$ & $0$ & $-a$ & $0$ & $0$ & $0$ & $0$ & $0$ & $0$ & $c$ & $0$& $0$\\
 & 4 & $bd$ & $-hb$ & $0$ & $0$ & $0$ & $0$ & $0$ & $0$ & $0$ & $-h$ & $1$ & $0$ & $0$ & $0$ & $0$ & $0$ & $0$ & $0$ & $0$ & $0$ & $0$ & $0$& $1$\\
 & 5 & $0$ & $0$ & $0$ & $0$ & $0$ & $0$ & $0$ & $0$ & $-b$ & $1$ & $0$ & $-b$ & $0$ & $0$ & $0$ & $-b$ & $1$ & $0$ & $0$ & $0$ & $0$ & $0$& $0$\\
 & 6 & $0$ & $0$ & $1$ & $0$ & $0$ & $0$ & $0$ & $0$ & $0$ & $0$ & $0$ & $-bc$ & $0$ & $1$ & $0$ & $-bc$ & $c$ & $0$ & $\frac{c}{a}$ & $0$ & $0$ & $0$& $0$\\
 & 7 & $-d$ & $0$ & $0$ & $0$ & $0$ & $0$ & $0$ & $0$ & $0$ & $0$ & $0$ & $0$ & $k$ & $0$ & $0$ & $0$ & $0$ & $-h$ & $0$ & $0$ & $0$ & $0$& $0$\\
 & 8 & $0$ & $1$ & $0$ & $0$ & $0$ & $1$ & $0$ & $0$ & $1$ & $0$ & $0$ & $1$ & $0$ & $0$ & $0$ & $1$ & $0$ & $1$ & $0$ & $0$ & $0$ & $0$& $0$\\
 & 9 & $0$ & $0$ & $0$ & $0$ & $0$ & $0$ & $0$ & $0$ & $0$ & $0$ & $0$ & $0$ & $1$ & $0$ & $1$ & $c$ & $0$ & $0$ & $0$ & $1$ & $-cd$ & $0$& $0$\\
\cmidrule{1-25}
\multirow{9}{*}{\begin{sideways}$V$ \end{sideways}}  & 1 & $0$ & $0$ & $1$ & $-f$ & $0$ & $0$ & $0$ & $0$ & $0$ & $0$ & $0$ & $0$ & $0$ & $1$ & $\frac{-da}{p}$ & $0$ & $0$ & $0$ & $1$ & $d$ & $1$ & $0$& $0$\\
 & 2 & $0$ & $0$ & $a$ & $0$ & $0$ & $0$ & $0$ & $0$ & $0$ & $0$ & $0$ & $0$ & $0$ & $0$ & $0$ & $0$ & $-f$ & $0$ & $0$ & $0$ & $0$ & $0$& $1$\\
 & 3 & $0$ & $0$ & $0$ & $0$ & $0$ & $0$ & $0$ & $0$ & $0$ & $0$ & $0$ & $-bf$ & $1$ & $0$ & $0$ & $-f$ & $-bf$ & $0$ & $0$ & $1$ & $0$ & $0$& $0$\\
 & 4 & $0$ & $0$ & $0$ & $1$ & $0$ & $0$ & $g$ & $0$ & $0$ & $0$ & $0$ & $0$ & $0$ & $0$ & $0$ & $0$ & $0$ & $0$ & $0$ & $0$ & $0$ & $1$& $0$\\
 & 5 & $0$ & $\frac{g}{bd}$ & $0$ & $0$ & $0$ & $1$ & $0$ & $0$ & $1$ & $\frac{-g}{d}$ & $0$ & $1$ & $0$ & $0$ & $0$ & $0$ & $1$ & $0$ & $0$ & $0$ & $0$ & $a$& $0$\\
 & 6 & $0$ & $0$ & $0$ & $0$ & $0$ & $b$ & $0$ & $0$ & $0$ & $0$ & $0$ & $b$ & $0$ & $0$ & $0$ & $1$ & $b$ & $\frac{-g}{d}$ & $0$ & $0$ & $0$ & $0$& $0$\\
 & 7 & $0$ & $0$ & $\frac{p}{a}$ & $0$ & $1$ & $0$ & $0$ & $d$ & $0$ & $0$ & $0$ & $0$ & $0$ & $\frac{p}{a}$ & $0$ & $0$ & $0$ & $0$ & $\frac{p}{a}$ & $0$ & $0$ & $0$& $0$\\
 & 8 & $0$ & $0$ & $p $ & $0$ & $a$ & $0$ & $0$ & $0$ & $0$ & $1$ & $1$ & $0$ & $0$ & $0$ & $0$ & $0$ & $0$ & $0$ & $p $ & $0$ & $0$ & $0$& $0$\\
 & 9 & $1$ & $1$ & $0$ & $0$ & $0$ & $0$ & $1$ & $1$ & $0$ & $0$ & $b$ & $0$ & $\frac{p}{a}$ & $0$ & $1$ & $0$ & $0$ & $1$ & $0$ & $0$ & $0$ & $0$& $\frac{-bp}{a}$\\
\cmidrule{1-25}
\multirow{9}{*}{\begin{sideways}$W$ \end{sideways}} &  1 & $\frac{p}{ad}$ & $0$ & $0$ & $0$ & $0$ & $0$ & $0$ & $0$ & $0$ & $0$ & $0$ & $0$ & $\frac{1}{k}$ & $0$ & $\frac{-p}{ak}$ & $0$ & $0$ & $0$ & $0$ & $\frac{-1}{k}$ & $0$ & $0$& $1$\\
 & 2 & $\frac{-1}{g}$ & $\frac{-d}{gh}$ & $0$ & $0$ & $0$ & $\frac{1}{bh}$ & $\frac{1}{g}$ & $0$ & $0$ & $0$ & $0$ & $0$ & $0$ & $0$ & $0$ & $0$ & $0$ & $\frac{d}{gh}$ & $0$ & $0$ & $0$ & $0$& $0$\\
 & 3 & $\frac{-1}{d}$ & $0$ & $0$ & $0$ & $0$ & $0$ & $0$ & $\frac{1}{d}$ & $0$ & $0$ & $1$ & $0$ & $0$ & $0$ & $0$ & $0$ & $0$ & $0$ & $0$ & $0$ & $0$ & $0$& $0$\\
 & 4 & $0$ & $0$ & $\frac{-c}{a}$ & $0$ & $\frac{-p}{a}$ & $\frac{1}{bf}$ & $0$ & $0$ & $0$ & $0$ & $0$ & $\frac{-1}{bf}$ & $0$ & $0$ & $0$ & $0$ & $\frac{-1}{f}$ & $0$ & $1$ & $0$ & $0$ & $0$& $0$\\
 & 5 & $0$ & $0$ & $0$ & $0$ & $0$ & $\frac{1}{b}$ & $0$ & $0$ & $\frac{-1}{b}$ & $0$ & $0$ & $0$ & $0$ & $0$ & $0$ & $0$ & $0$ & $0$ & $0$ & $0$ & $0$ & $1$& $0$\\
 & 6 & $0$ & $1$ & $0$ & $0$ & $1$ & $0$ & $0$ & $0$ & $\frac{-g}{bd}$ & $1$ & $h$ & $0$ & $0$ & $0$ & $0$ & $0$ & $0$ & $0$ & $0$ & $0$ & $0$ & $0$& $0$\\
 & 7 & $0$ & $0$ & $\frac{1}{a}$ & $0$ & $\frac{p}{ac}$ & $0$ & $0$ & $0$ & $0$ & $0$ & $0$ & $0$ & $0$ & $0$ & $0$ & $0$ & $0$ & $0$ & $\frac{-1}{c}$ & $1$ & $\frac{1}{c}$ & $0$& $0$\\
 & 8 & $0$ & $0$ & $0$ & $\frac{1}{c}$ & $0$ & $0$ & $0$ & $0$ & $\frac{1}{bc}$ & $0$ & $0$ & $\frac{-1}{bc}$ & $0$ & $0$ & $0$ & $\frac{1}{c}$ & $0$ & $0$ & $0$ & $f$ & $\frac{f}{c}$ & $\frac{-1}{c}$& $0$\\
 & 9 & $0$ & $0$ & $0$ & $0$ & $\frac{-1}{c}$ & $0$ & $0$ & $\frac{-k}{d}$ & $0$ & $0$ & $0$ & $0$ & $0$ & $\frac{-1}{p}$ & $1$ & $0$ & $0$ & $0$ & $\frac{a}{cp}$ & $0$ & $\frac{-a}{cp}$ & $0$& $0$\\ 
\bottomrule
\end{tabular}
\end{minipage}}
\end{sidewaystable}
 
\paragraph*{Rank 47 of \mbox{\boldmath $T_{345}$ and permutations of $m$, $p$, and $n$}.}
We obtain the three-parameter family of PDs in \Cref{tab:param_R47_U_1,tab:param_R47_U_2,tab:param_R47_V,tab:param_R47_W} in the same way as the nine-parameter family in \Cref{tab:9_param_Smirnov} but instead of the PD from Smirnov, we use the discrete PD of rank 47 recently found by DeepMind \cite{DeepMind_2022} in the starting point. The rank of the Jacobian at the PDs in this family is 2051, which is the same as for $x^*$. The stability improvements and different optimal parameters are given in \Cref{tab:results_param}. We improve both factors for all permutations of $m$, $p$, and $n$. The PDs in this family have 14 nonzeros less compared to the PD from DeepMind, which additionally decreases the constant $c$ in the asymptotic complexity.

\begin{table}[ht]
\centering
\caption{Matrix $U$ of the three-parameter family of $\mathrm{PD}_{47}(T_{345})$s with parameters $a$, $b$, and $c$, part 1.}
\label{tab:param_R47_U_1}
\resizebox{\textwidth}{!}{\begin{minipage}{\textwidth}
\begin{tabular}{c c | c c c c c c c c c c c c c c c c c}
 & & 1 & 2 & 3 & 4 & 5 & 6 & 7 & 8 & 9 & 10 & 11 & 12 & 13 & 14 & 15 & 16 \\
\cmidrule{1-18}
 \multirow{41}{*}{\begin{sideways}$U(:,1:32)$ \end{sideways}}  & 1 & $0$ & $0$ & $0$ & $0$ & $0$ & $0$ & $0$ & $0$ & $0$ & $0$ & $0$ & $1$ & $1$ & $0$ & $0$ & $0$\\
 & 2 & $0$ & $0$ & $0$ & $0$ & $0$ & $0$ & $0$ & $-1$ & $0$ & $-1$ & $1/a$ & $0$ & $0$ & $0$ & $0$ & $0$\\
 & 3 & $0$ & $0$ & $1$ & $0$ & $0$ & $0$ & $0$ & $0$ & $0$ & $0$ & $0$ & $0$ & $1$ & $0$ & $0$ & $0$\\
 & 4 & $0$ & $0$ & $-1$ & $0$ & $1$ & $0$ & $0$ & $0$ & $0$ & $0$ & $0$ & $-1$ & $-1$ & $1$ & $0$ & $1$\\
 & 5 & $0$ & $0$ & $-1$ & $-1$ & $0$ & $0$ & $0$ & $0$ & $0$ & $0$ & $0$ & $-1$ & $-1$ & $0$ & $0$ & $0$\\
 & 6 & $0$ & $0$ & $0$ & $0$ & $0$ & $0$ & $0$ & $0$ & $0$ & $0$ & $1$ & $1$ & $1$ & $0$ & $0$ & $0$\\
 & 7 & $0$ & $0$ & $0$ & $0$ & $0$ & $0$ & $0$ & $-1$ & $0$ & $-1$ & $0$ & $0$ & $0$ & $0$ & $0$ & $0$\\
 & 8 & $1$ & $0$ & $0$ & $0$ & $0$ & $0$ & $0$ & $0$ & $0$ & $-1$ & $0$ & $0$ & $0$ & $0$ & $0$ & $0$\\
 & 9 & $-1$ & $0$ & $0$ & $0$ & $0$ & $0$ & $0$ & $1$ & $0$ & $1$ & $0$ & $0$ & $0$ & $0$ & $0$ & $0$\\
 & 10 & $-1$ & $1$ & $0$ & $0$ & $0$ & $0$ & $0$ & $1$ & $1$ & $1$ & $0$ & $0$ & $0$ & $0$ & $1$ & $0$\\
 & 11 & $0$ & $0$ & $0$ & $0$ & $0$ & $0$ & $0$ & $0$ & $0$ & $0$ & $0$ & $0$ & $0$ & $1$ & $0$ & $0$\\
 & 12 & $0$ & $0$ & $0$ & $0$ & $0$ & $0$ & $0$ & $0$ & $0$ & $0$ & $0$ & $0$ & $0$ & $0$ & $0$ & $0$\\
 & 13 & $0$ & $0$ & $0$ & $0$ & $1$ & $1$ & $0$ & $0$ & $0$ & $0$ & $0$ & $0$ & $0$ & $0$ & $0$ & $0$\\
 & 14 & $0$ & $0$ & $0$ & $0$ & $0$ & $0$ & $0$ & $0$ & $0$ & $0$ & $0$ & $0$ & $0$ & $0$ & $0$ & $0$\\
 & 15 & $0$ & $0$ & $0$ & $0$ & $0$ & $0$ & $0$ & $0$ & $0$ & $0$ & $0$ & $0$ & $0$ & $-1$ & $0$ & $0$\\
 & 16 & $0$ & $0$ & $0$ & $0$ & $0$ & $0$ & $0$ & $0$ & $0$ & $0$ & $0$ & $0$ & $0$ & $0$ & $0$ & $0$\\
 & 17 & $0$ & $0$ & $0$ & $0$ & $0$ & $0$ & $0$ & $0$ & $-1$ & $0$ & $0$ & $0$ & $0$ & $0$ & $0$ & $0$\\
 & 18 & $0$ & $-1$ & $0$ & $1$ & $0$ & $0$ & $-1$ & $0$ & $0$ & $0$ & $0$ & $0$ & $0$ & $0$ & $0$ & $0$\\
 & 19 & $0$ & $0$ & $0$ & $0$ & $0$ & $0$ & $0$ & $0$ & $1$ & $0$ & $0$ & $0$ & $0$ & $0$ & $0$ & $0$\\
 & 20 & $0$ & $0$ & $0$ & $0$ & $0$ & $0$ & $0$ & $0$ & $0$ & $0$ & $0$ & $0$ & $0$ & $0$ & $0$ & $0$\\
\cmidrule{2-18}
& & 17 & 18 & 19 & 20 & 21 & 22 & 23 & 24 & 25 & 26 & 27 & 28 & 29 & 30 & 31 & 32\\
\cmidrule{2-18}
& 1 & $0$ & $0$ & $0$ & $0$ & $0$ & $0$ & $0$ & $1$ & $1$ & $0$ & $0$ & $0$ & $0$ & $0$ & $1$ & $0$\\
 & 2 & $0$ & $0$ & $0$ & $0$ & $0$ & $0$ & $0$ & $0$ & $0$ & $0$ & $-1$ & $0$ & $1$ & $0$ & $-1/a$ & $0$\\
 & 3 & $0$ & $0$ & $0$ & $0$ & $0$ & $0$ & $0$ & $1$ & $1$ & $0$ & $-1$ & $0$ & $0$ & $0$ & $0$ & $0$\\
 & 4 & $0$ & $0$ & $0$ & $0$ & $0$ & $0$ & $0$ & $-1$ & $-1$ & $0$ & $0$ & $0$ & $0$ & $0$ & $0$ & $0$\\
 & 5 & $-1$ & $-1$ & $0$ & $0$ & $0$ & $0$ & $0$ & $-1$ & $-1$ & $0$ & $0$ & $0$ & $0$ & $0$ & $0$ & $0$\\
 & 6 & $0$ & $0$ & $0$ & $0$ & $0$ & $0$ & $0$ & $1$ & $0$ & $-a$ & $0$ & $0$ & $0$ & $0$ & $0$ & $0$\\
 & 7 & $0$ & $0$ & $0$ & $0$ & $0$ & $0$ & $0$ & $0$ & $0$ & $1$ & $-1$ & $0$ & $0$ & $0$ & $0$ & $0$\\
 & 8 & $0$ & $0$ & $0$ & $0$ & $0$ & $0$ & $0$ & $1$ & $0$ & $0$ & $-1$ & $0$ & $0$ & $0$ & $0$ & $0$\\
 & 9 & $0$ & $0$ & $1$ & $1$ & $0$ & $0$ & $0$ & $0$ & $0$ & $0$ & $1$ & $0$ & $0$ & $1$ & $0$ & $0$\\
 & 10 & $0$ & $0$ & $0$ & $0$ & $0$ & $0$ & $0$ & $0$ & $0$ & $0$ & $1$ & $0$ & $0$ & $0$ & $0$ & $0$\\
 & 11 & $0$ & $0$ & $0$ & $0$ & $0$ & $0$ & $0$ & $0$ & $0$ & $0$ & $0$ & $0$ & $0$ & $1$ & $1$ & $0$\\
 & 12 & $0$ & $0$ & $0$ & $-1$ & $0$ & $1$ & $0$ & $0$ & $0$ & $-1$ & $0$ & $1$ & $0$ & $0$ & $0$ & $0$\\
 & 13 & $0$ & $0$ & $-1$ & $0$ & $0$ & $0$ & $0$ & $-1$ & $0$ & $0$ & $-1$ & $0$ & $0$ & $0$ & $0$ & $0$\\
 & 14 & $0$ & $0$ & $0$ & $0$ & $0$ & $0$ & $0$ & $0$ & $0$ & $0$ & $0$ & $0$ & $0$ & $-1$ & $0$ & $0$\\
 & 15 & $0$ & $-1$ & $0$ & $0$ & $0$ & $0$ & $1$ & $0$ & $0$ & $0$ & $0$ & $0$ & $0$ & $-1$ & $0$ & $0$\\
 & 16 & $0$ & $1$ & $0$ & $0$ & $0$ & $0$ & $0$ & $0$ & $0$ & $0$ & $0$ & $0$ & $0$ & $0$ & $-1$ & $0$\\
 & 17 & $0$ & $0$ & $0$ & $0$ & $1$ & $0$ & $0$ & $0$ & $0$ & $-1$ & $0$ & $1$ & $0$ & $0$ & $0$ & $0$\\
 & 18 & $0$ & $0$ & $0$ & $0$ & $0$ & $0$ & $0$ & $-1$ & $0$ & $0$ & $1$ & $0$ & $0$ & $0$ & $0$ & $0$\\
 & 19 & $0$ & $0$ & $0$ & $-1$ & $-1$ & $1$ & $0$ & $0$ & $0$ & $0$ & $0$ & $0$ & $0$ & $-1$ & $0$ & $-1$\\
 & 20 & $0$ & $0$ & $0$ & $0$ & $0$ & $0$ & $0$ & $0$ & $0$ & $0$ & $0$ & $0$ & $0$ & $0$ & $0$ & $0$\\
\bottomrule
\end{tabular}
\end{minipage}}
\end{table}

\begin{table}[ht]
\centering
\caption{Matrix $U$ of the three-parameter family of $\mathrm{PD}_{47}(T_{345})$s with parameters $a$, $b$, and $c$, part 2.}
\label{tab:param_R47_U_2}
\resizebox{\textwidth}{!}{\begin{minipage}{\textwidth}
\begin{tabular}{c c | c c c c c c c c c c c c c c c c c}
 & &  33 & 34 & 35 & 36 & 37 & 38 & 39 & 40 & 41 & 42 & 43 & 44 & 45 & 46 & 47 \\
\cmidrule{1-17}
 \multirow{20}{*}{\begin{sideways}$U(:,33:47)$ \end{sideways}}  
& 1 & $0$ & $0$ & $0$ & $0$ & $0$ & $0$ & $0$ & $0$ & $0$ & $0$ & $0$ & $0$ & $0$ & $0$ & $0$\\
 & 2 & $0$ & $0$ & $0$ & $0$ & $0$ & $0$ & $0$ & $0$ & $0$ & $0$ & $0$ & $0$ & $0$ & $0$ & $0$\\
 & 3 & $0$ & $0$ & $0$ & $0$ & $0$ & $0$ & $0$ & $0$ & $0$ & $0$ & $0$ & $0$ & $0$ & $0$ & $0$\\
 & 4 & $0$ & $1$ & $0$ & $0$ & $0$ & $0$ & $0$ & $0$ & $0$ & $0$ & $0$ & $0$ & $0$ & $0$ & $0$\\
 & 5 & $0$ & $0$ & $-1$ & $0$ & $0$ & $0$ & $0$ & $0$ & $0$ & $0$ & $0$ & $0$ & $0$ & $0$ & $0$\\
 & 6 & $0$ & $0$ & $0$ & $0$ & $0$ & $0$ & $0$ & $0$ & $0$ & $0$ & $0$ & $0$ & $0$ & $0$ & $-1$\\
 & 7 & $0$ & $0$ & $0$ & $0$ & $0$ & $0$ & $0$ & $0$ & $0$ & $1$ & $0$ & $0$ & $0$ & $0$ & $0$\\
 & 8 & $0$ & $0$ & $0$ & $0$ & $0$ & $0$ & $0$ & $0$ & $0$ & $1$ & $0$ & $0$ & $0$ & $0$ & $0$\\
 & 9 & $0$ & $0$ & $0$ & $1$ & $0$ & $0$ & $0$ & $0$ & $0$ & $-1$ & $0$ & $0$ & $0$ & $0$ & $0$\\
 & 10 & $1$ & $0$ & $0$ & $0$ & $0$ & $0$ & $0$ & $0$ & $0$ & $-1$ & $0$ & $0$ & $0$ & $0$ & $0$\\
 & 11 & $0$ & $0$ & $0$ & $0$ & $0$ & $0$ & $0$ & $0$ & $0$ & $0$ & $1$ & $0$ & $1$ & $0$ & $0$\\
 & 12 & $0$ & $-1$ & $0$ & $0$ & $0$ & $0$ & $0$ & $1$ & $0$ & $0$ & $0$ & $0$ & $0$ & $0$ & $0$\\
 & 13 & $0$ & $0$ & $0$ & $0$ & $0$ & $0$ & $0$ & $0$ & $-1$ & $0$ & $0$ & $0$ & $0$ & $0$ & $0$\\
 & 14 & $0$ & $1$ & $0$ & $0$ & $0$ & $0$ & $0$ & $0$ & $1$ & $0$ & $0$ & $0$ & $0$ & $0$ & $0$\\
 & 15 & $0$ & $0$ & $-1$ & $0$ & $0$ & $0$ & $1$ & $0$ & $0$ & $0$ & $0$ & $1$ & $-1$ & $-1$ & $0$\\
 & 16 & $1$ & $0$ & $0$ & $0$ & $0$ & $1$ & $-1$ & $0$ & $0$ & $0$ & $-1$ & $0$ & $0$ & $0$ & $0$\\
 & 17 & $0$ & $0$ & $-1$ & $0$ & $0$ & $0$ & $0$ & $0$ & $0$ & $0$ & $0$ & $0$ & $0$ & $0$ & $0$\\
 & 18 & $0$ & $0$ & $0$ & $0$ & $0$ & $0$ & $0$ & $0$ & $0$ & $0$ & $0$ & $1$ & $0$ & $0$ & $0$\\
 & 19 & $0$ & $0$ & $1$ & $0$ & $-1$ & $0$ & $0$ & $0$ & $1$ & $0$ & $0$ & $0$ & $0$ & $0$ & $0$\\
 & 20 & $-1$ & $0$ & $1$ & $0$ & $0$ & $0$ & $0$ & $0$ & $0$ & $0$ & $0$ & $-1$ & $0$ & $0$ & $0$\\
\bottomrule
\end{tabular}
\end{minipage}}
\end{table}

\begin{table}[ht]
\centering
\caption{Matrix $V$ of the three-parameter family of $\mathrm{PD}_{47}(T_{345})$s with parameters $a$, $b$, and $c$.}
\label{tab:param_R47_V}
\resizebox{\textwidth}{!}{\begin{minipage}{\textwidth}
\begin{tabular}{c c | c c c c c c c c c c c c c c c c c}
 & & 1 & 2 & 3 & 4 & 5 & 6 & 7 & 8 & 9 & 10 & 11 & 12 & 13 & 14 & 15 & 16 \\
 \cmidrule{2-18}
\multirow{38}{*}{\begin{sideways}$V$ \end{sideways}}  & 1 & $0$ & $0$ & $0$ & $1$ & $-1$ & $1$ & $1$ & $0$ & $0$ & $1$ & $0$ & $0$ & $0$ & $-1$ & $0$ & $1$\\
 & 2 & $0$ & $0$ & $0$ & $0$ & $0$ & $0$ & $0$ & $-1$ & $-1$ & $-1$ & $-1$ & $0$ & $0$ & $0$ & $0$ & $0$\\
 & 3 & $0$ & $0$ & $0$ & $0$ & $1$ & $-1$ & $0$ & $0$ & $0$ & $0$ & $0$ & $0$ & $0$ & $1$ & $0$ & $-1$\\
 & 4 & $0$ & $0$ & $0$ & $1$ & $0$ & $0$ & $1$ & $0$ & $0$ & $0$ & $0$ & $0$ & $0$ & $0$ & $0$ & $0$\\
 & 5 & $0$ & $0$ & $-1$ & $-1$ & $1$ & $-1$ & $-1$ & $0$ & $0$ & $-1$ & $0$ & $0$ & $0$ & $0$ & $0$ & $0$\\
 & 6 & $-1$ & $1$ & $0$ & $0$ & $0$ & $1$ & $-1$ & $0$ & $0$ & $0$ & $0$ & $0$ & $-1$ & $0$ & $0$ & $0$\\
 & 7 & $0$ & $0$ & $0$ & $0$ & $0$ & $1$ & $0$ & $0$ & $0$ & $0$ & $0$ & $0$ & $0$ & $0$ & $0$ & $0$\\
 & 8 & $0$ & $0$ & $0$ & $0$ & $0$ & $0$ & $-1$ & $0$ & $0$ & $0$ & $0$ & $0$ & $0$ & $0$ & $0$ & $0$\\
 & 9 & $0$ & $0$ & $0$ & $0$ & $0$ & $0$ & $0$ & $0$ & $0$ & $0$ & $1$ & $1$ & $1$ & $1$ & $0$ & $0$\\
 & 10 & $0$ & $1$ & $0$ & $0$ & $0$ & $1$ & $-1$ & $0$ & $1$ & $0$ & $0$ & $0$ & $-1$ & $0$ & $1$ & $0$\\
 & 11 & $0$ & $0$ & $0$ & $0$ & $0$ & $1$ & $0$ & $0$ & $0$ & $0$ & $0$ & $0$ & $0$ & $0$ & $0$ & $0$\\
 & 12 & $0$ & $1$ & $0$ & $0$ & $0$ & $0$ & $-1$ & $0$ & $1$ & $0$ & $0$ & $0$ & $0$ & $0$ & $1$ & $0$\\
\cmidrule{2-18}
& & 17 & 18 & 19 & 20 & 21 & 22 & 23 & 24 & 25 & 26 & 27 & 28 & 29 & 30 & 31 & 32\\
\cmidrule{2-18}
 & 1 & $1$ & $1$ & $0$ & $0$ & $0$ & $0$ & $0$ & $0$ & $0$ & $0$ & $1$ & $0$ & $1$ & $0$ & $1$ & $0$\\
 & 2 & $0$ & $0$ & $0$ & $-1$ & $-1$ & $-1$ & $0$ & $0$ & $0$ & $-1$ & $0$ & $\frac{-1}{b}$ & $-1$ & $0$ & $0$ & $0$\\
 & 3 & $0$ & $0$ & $0$ & $0$ & $0$ & $0$ & $0$ & $0$ & $0$ & $0$ & $0$ & $0$ & $0$ & $0$ & $0$ & $0$\\
 & 4 & $1$ & $1$ & $0$ & $0$ & $-1$ & $-1$ & $0$ & $0$ & $0$ & $0$ & $0$ & $\frac{-1}{b}$ & $0$ & $0$ & $0$ & $0$\\
 & 5 & $0$ & $0$ & $0$ & $0$ & $0$ & $0$ & $0$ & $0$ & $-1$ & $0$ & $-1$ & $0$ & $0$ & $0$ & $0$ & $0$\\
 & 6 & $0$ & $0$ & $-1$ & $0$ & $0$ & $0$ & $0$ & $1$ & $1$ & $0$ & $0$ & $0$ & $0$ & $0$ & $0$ & $0$\\
 & 7 & $0$ & $0$ & $0$ & $0$ & $0$ & $b$ & $1$ & $0$ & $0$ & $0$ & $0$ & $0$ & $0$ & $0$ & $0$ & $b$\\
 & 8 & $0$ & $0$ & $0$ & $0$ & $0$ & $-b$ & $1$ & $0$ & $0$ & $0$ & $0$ & $-1$ & $0$ & $0$ & $0$ & $-b$\\
 & 9 & $0$ & $-1$ & $0$ & $0$ & $0$ & $0$ & $0$ & $0$ & $-1$ & $0$ & $0$ & $0$ & $0$ & $0$ & $-1$ & $0$\\
 & 10 & $0$ & $0$ & $-1$ & $1$ & $1$ & $1$ & $0$ & $1$ & $1$ & $1$ & $0$ & $\frac{1}{b}$ & $0$ & $0$ & $0$ & $0$\\
 & 11 & $0$ & $0$ & $-1$ & $1$ & $0$ & $0$ & $0$ & $0$ & $0$ & $0$ & $0$ & $0$ & $0$ & $1$ & $0$ & $-1$\\
 & 12 & $0$ & $0$ & $0$ & $0$ & $1$ & $1$ & $0$ & $0$ & $0$ & $0$ & $0$ & $\frac{1}{b}$ & $0$ & $0$ & $0$ & $1$\\
\cmidrule{2-18}
 & &  33 & 34 & 35 & 36 & 37 & 38 & 39 & 40 & 41 & 42 & 43 & 44 & 45 & 46 & 47 \\
 \cmidrule{2-17}
& 1 & $0$ & $0$ & $0$ & $0$ & $0$ & $0$ & $c$ & $0$ & $0$ & $1$ & $1$ & $0$ & $1$ & $0$ & $0$\\
 & 2 & $0$ & $0$ & $0$ & $0$ & $0$ & $0$ & $0$ & $0$ & $0$ & $-1$ & $0$ & $0$ & $0$ & $0$ & $0$\\
 & 3 & $0$ & $1$ & $0$ & $0$ & $0$ & $1$ & $-c$ & $1$ & $0$ & $0$ & $-1$ & $0$ & $-1$ & $1$ & $0$\\
 & 4 & $0$ & $0$ & $1$ & $0$ & $0$ & $1$ & $0$ & $-1$ & $0$ & $0$ & $0$ & $0$ & $0$ & $1$ & $0$\\
 & 5 & $0$ & $0$ & $0$ & $0$ & $0$ & $0$ & $0$ & $0$ & $0$ & $-1$ & $0$ & $0$ & $0$ & $0$ & $0$\\
 & 6 & $0$ & $0$ & $0$ & $0$ & $0$ & $0$ & $0$ & $0$ & $0$ & $1$ & $0$ & $0$ & $0$ & $0$ & $0$\\
 & 7 & $0$ & $0$ & $0$ & $0$ & $0$ & $\frac{1}{c}$ & $-1$ & $b$ & $1$ & $0$ & $\frac{-1}{c}$ & $0$ & $0$ & $0$ & $0$\\
 & 8 & $0$ & $0$ & $0$ & $0$ & $0$ & $\frac{1}{c}$ & $-1$ & $-b$ & $0$ & $0$ & $0$ & $1$ & $0$ & $0$ & $0$\\
 & 9 & $0$ & $0$ & $0$ & $0$ & $0$ & $0$ & $-c$ & $0$ & $0$ & $0$ & $-1$ & $0$ & $-1$ & $0$ & $1$\\
 & 10 & $0$ & $0$ & $0$ & $1$ & $0$ & $0$ & $0$ & $0$ & $0$ & $0$ & $0$ & $0$ & $0$ & $0$ & $-1$\\
 & 11 & $0$ & $0$ & $0$ & $1$ & $1$ & $-1$ & $c$ & $-1$ & $0$ & $0$ & $1$ & $0$ & $1$ & $0$ & $0$\\
 & 12 & $1$ & $0$ & $0$ & $0$ & $-1$ & $-1$ & $0$ & $1$ & $0$ & $0$ & $0$ & $0$ & $0$ & $0$ & $0$\\
\bottomrule
\end{tabular}
\end{minipage}}
\end{table}

\begin{table}[htbp]
\centering
\caption{Matrix $W$ of the three-parameter family of $\mathrm{PD}_{47}(T_{345})$s with parameters $a$, $b$, and $c$.}
\label{tab:param_R47_W}
\resizebox{\textwidth}{!}{\begin{minipage}{\textwidth}
\begin{tabular}{c c | c c c c c c c c c c c c c c c c c}
 & & 1 & 2 & 3 & 4 & 5 & 6 & 7 & 8 & 9 & 10 & 11 & 12 & 13 & 14 & 15 & 16 \\
 \cmidrule{2-18}
\multirow{48}{*}{\begin{sideways}$W$ \end{sideways}}   & 1 & $0$ & $0$ & $0$ & $0$ & $0$ & $0$ & $0$ & $0$ & $0$ & $0$ & $-1$ & $1$ & $0$ & $1$ & $0$ & $1$\\
 & 2 & $0$ & $0$ & $1$ & $0$ & $0$ & $0$ & $0$ & $0$ & $0$ & $0$ & $0$ & $0$ & $-1$ & $0$ & $0$ & $0$\\
 & 3 & $0$ & $0$ & $0$ & $0$ & $0$ & $0$ & $0$ & $0$ & $0$ & $0$ & $0$ & $1$ & $0$ & $1$ & $0$ & $1$\\
 & 4 & $0$ & $0$ & $0$ & $0$ & $0$ & $0$ & $0$ & $1$ & $-1$ & $0$ & $0$ & $0$ & $0$ & $0$ & $1$ & $0$\\
 & 5 & $1$ & $0$ & $0$ & $0$ & $0$ & $0$ & $0$ & $0$ & $0$ & $1$ & $0$ & $0$ & $0$ & $0$ & $0$ & $0$\\
 & 6 & $0$ & $0$ & $0$ & $0$ & $0$ & $0$ & $0$ & $1$ & $-1$ & $0$ & $a$ & $0$ & $0$ & $0$ & $1$ & $0$\\
 & 7 & $0$ & $0$ & $-1$ & $1$ & $1$ & $0$ & $0$ & $-1$ & $0$ & $1$ & $0$ & $0$ & $0$ & $0$ & $0$ & $1$\\
 & 8 & $-1$ & $1$ & $-1$ & $1$ & $1$ & $1$ & $1$ & $0$ & $0$ & $0$ & $0$ & $0$ & $0$ & $0$ & $-1$ & $1$\\
 & 9 & $1$ & $-1$ & $0$ & $0$ & $0$ & $0$ & $0$ & $0$ & $0$ & $0$ & $0$ & $-1$ & $1$ & $0$ & $1$ & $0$\\
 & 10 & $0$ & $0$ & $0$ & $0$ & $0$ & $0$ & $0$ & $0$ & $0$ & $0$ & $0$ & $0$ & $0$ & $0$ & $0$ & $1$\\
 & 11 & $0$ & $0$ & $0$ & $0$ & $1$ & $1$ & $0$ & $0$ & $0$ & $0$ & $0$ & $0$ & $0$ & $0$ & $0$ & $1$\\
 & 12 & $0$ & $0$ & $0$ & $0$ & $0$ & $0$ & $0$ & $0$ & $0$ & $0$ & $0$ & $0$ & $0$ & $1$ & $0$ & $1$\\
 & 13 & $0$ & $0$ & $0$ & $0$ & $0$ & $0$ & $0$ & $0$ & $-1$ & $0$ & $0$ & $0$ & $0$ & $0$ & $1$ & $0$\\
 & 14 & $0$ & $1$ & $0$ & $1$ & $0$ & $0$ & $1$ & $0$ & $0$ & $0$ & $0$ & $0$ & $0$ & $0$ & $-1$ & $0$\\
 & 15 & $0$ & $0$ & $0$ & $0$ & $0$ & $0$ & $0$ & $0$ & $0$ & $0$ & $0$ & $0$ & $0$ & $0$ & $1$ & $0$\\
\cmidrule{2-18}
& & 17 & 18 & 19 & 20 & 21 & 22 & 23 & 24 & 25 & 26 & 27 & 28 & 29 & 30 & 31 & 32\\
\cmidrule{2-18}
 & 1 & $-1$ & $1$ & $0$ & $0$ & $0$ & $0$ & $0$ & $0$ & $0$ & $0$ & $0$ & $0$ & $\frac{1}{a}$ & $0$ & $1$ & $0$\\
 & 2 & $0$ & $0$ & $0$ & $0$ & $0$ & $0$ & $1$ & $0$ & $-1$ & $0$ & $0$ & $0$ & $0$ & $0$ & $0$ & $0$\\
 & 3 & $-1$ & $1$ & $0$ & $0$ & $0$ & $0$ & $\frac{1}{c} $ & $0$ & $0$ & $0$ & $0$ & $0$ & $0$ & $0$ & $0$ & $0$\\
 & 4 & $0$ & $0$ & $0$ & $-1$ & $-1$ & $-1$ & $0$ & $0$ & $0$ & $0$ & $0$ & $0$ & $1$ & $0$ & $0$ & $-1$\\
 & 5 & $0$ & $0$ & $0$ & $0$ & $\frac{1}{b}$ & $\frac{1}{b}$ & $0$ & $0$ & $0$ & $0$ & $0$ & $-1$ & $1$ & $0$ & $0$ & $\frac{1}{b}$\\
 & 6 & $0$ & $0$ & $0$ & $-1$ & $0$ & $0$ & $0$ & $0$ & $0$ & $1$ & $0$ & $0$ & $0$ & $0$ & $0$ & $0$\\
 & 7 & $-1$ & $0$ & $0$ & $0$ & $0$ & $0$ & $0$ & $0$ & $0$ & $0$ & $-1$ & $0$ & $0$ & $0$ & $0$ & $0$\\
 & 8 & $-1$ & $0$ & $-1$ & $0$ & $0$ & $0$ & $0$ & $0$ & $0$ & $0$ & $0$ & $0$ & $0$ & $0$ & $0$ & $0$\\
 & 9 & $0$ & $0$ & $1$ & $0$ & $0$ & $0$ & $0$ & $1$ & $0$ & $0$ & $0$ & $0$ & $0$ & $0$ & $0$ & $0$\\
 & 10 & $0$ & $0$ & $0$ & $-1$ & $0$ & $-1$ & $0$ & $0$ & $0$ & $0$ & $0$ & $0$ & $0$ & $0$ & $0$ & $-1$\\
 & 11 & $0$ & $0$ & $-1$ & $0$ & $0$ & $0$ & $0$ & $0$ & $0$ & $0$ & $0$ & $0$ & $0$ & $0$ & $0$ & $\frac{1}{b}$\\
 & 12 & $0$ & $0$ & $0$ & $0$ & $0$ & $0$ & $0$ & $0$ & $0$ & $0$ & $0$ & $0$ & $0$ & $-1$ & $0$ & $0$\\
 & 13 & $-1$ & $0$ & $0$ & $0$ & $-1$ & $0$ & $0$ & $0$ & $0$ & $0$ & $0$ & $0$ & $0$ & $0$ & $0$ & $0$\\
 & 14 & $-1$ & $0$ & $0$ & $0$ & $0$ & $0$ & $1$ & $0$ & $0$ & $0$ & $0$ & $0$ & $0$ & $0$ & $0$ & $0$\\
 & 15 & $-1$ & $1$ & $0$ & $0$ & $0$ & $0$ & $\frac{1}{c} $ & $0$ & $0$ & $0$ & $0$ & $0$ & $0$ & $0$ & $0$ & $0$\\
\cmidrule{2-18}
 & &  33 & 34 & 35 & 36 & 37 & 38 & 39 & 40 & 41 & 42 & 43 & 44 & 45 & 46 & 47 \\
 \cmidrule{2-18}
  & 1 & $0$ & $0$ & $0$ & $0$ & $0$ & $0$ & $0$ & $0$ & $0$ & $0$ & $0$ & $0$ & $0$ & $-1$ & $0$\\
 & 2 & $0$ & $0$ & $0$ & $0$ & $0$ & $0$ & $1$ & $0$ & $0$ & $0$ & $-c$ & $0$ & $c$ & $0$ & $-1$\\
 & 3 & $0$ & $0$ & $0$ & $0$ & $0$ & $-1$ & $\frac{1}{c}$ & $0$ & $0$ & $0$ & $0$ & $0$ & $1$ & $-1$ & $1$\\
 & 4 & $0$ & $0$ & $0$ & $1$ & $0$ & $0$ & $0$ & $1$ & $0$ & $0$ & $0$ & $0$ & $0$ & $0$ & $0$\\
 & 5 & $0$ & $0$ & $0$ & $0$ & $\frac{1}{b}$ & $0$ & $0$ & $0$ & $0$ & $1$ & $0$ & $0$ & $0$ & $0$ & $0$\\
 & 6 & $0$ & $0$ & $0$ & $1$ & $1$ & $0$ & $0$ & $0$ & $0$ & $0$ & $0$ & $0$ & $0$ & $0$ & $a$\\
 & 7 & $0$ & $0$ & $0$ & $0$ & $0$ & $0$ & $0$ & $0$ & $0$ & $0$ & $0$ & $0$ & $0$ & $0$ & $0$\\
 & 8 & $0$ & $0$ & $0$ & $-1$ & $0$ & $0$ & $0$ & $0$ & $0$ & $0$ & $0$ & $0$ & $0$ & $0$ & $0$\\
 & 9 & $0$ & $0$ & $0$ & $1$ & $0$ & $0$ & $0$ & $0$ & $0$ & $0$ & $0$ & $0$ & $0$ & $0$ & $0$\\
 & 10 & $0$ & $1$ & $0$ & $1$ & $0$ & $0$ & $0$ & $1$ & $0$ & $0$ & $0$ & $0$ & $0$ & $0$ & $0$\\
 & 11 & $0$ & $0$ & $0$ & $-1$ & $\frac{1}{b}$ & $0$ & $0$ & $0$ & $1$ & $0$ & $0$ & $0$ & $0$ & $0$ & $0$\\
 & 12 & $0$ & $0$ & $0$ & $1$ & $1$ & $0$ & $0$ & $0$ & $0$ & $0$ & $0$ & $0$ & $1$ & $0$ & $0$\\
 & 13 & $0$ & $0$ & $1$ & $0$ & $0$ & $0$ & $0$ & $0$ & $0$ & $0$ & $0$ & $0$ & $0$ & $-1$ & $0$\\
 & 14 & $0$ & $0$ & $0$ & $0$ & $0$ & $0$ & $0$ & $0$ & $0$ & $0$ & $0$ & $-1$ & $0$ & $0$ & $0$\\
 & 15 & $-1$ & $0$ & $0$ & $0$ & $0$ & $-1$ & $\frac{1}{c} $ & $0$ & $0$ & $0$ & $0$ & $0$ & $0$ & $-1$ & $0$\\
\bottomrule
\end{tabular}
\end{minipage}}
\end{table}

\paragraph*{Rank 49 of \mbox{\boldmath $T_{444}$}.}
In \cite{DeepMind_2022}, $14000$ non inv-equivalent discrete PDs of rank 49 are obtained with elements in $\lbrace 0, \pm 1, \pm 2 \rbrace$ using deep reinforcement learning. However, as explained at the beginning of this section, it is possible that some of these PDs belong to the same parameter family. Among those $14000$ PDs, they obtained a more stable PD compared to $\mathrm{PD}_{\mathrm{Strassen}}^{\mathrm{rec}}$. More specifically, they improved the stability and prefactor \eqref{eq:e_q} from $q=24$ and $e=144$ to $q=23$ and $e=136$. We obtain the
13-parameter family of PDs in \Cref{tab:param_R49_U,tab:param_R49_V,tab:param_R49_W} in the same way as the nine-parameter family in \Cref{tab:9_param_Smirnov} by using the most stable PD of rank 49 from DeepMind in the definition of the starting point.
Using this parametrization, we improve the pre- and stability factor to $q=22$ and $e=99$ for $(b,f,g,h,p,t):= \frac{1}{2}$, $(c,d,n,s,w):=1$, $a:=2$, and $v:=4$. The rank of the Jacobian at this PD is 2198.

In \Cref{tab:5param_R49_S13_A,tab:5param_R49_S13_BCD}, the submatrices $A$, $B$, $C$, and $D$ of a five-parameter family of CS PDs with $S=13$ and $T=12$ are given. Up to our knowledge, this is the first (parameter family of) CS PD(s) known for this combination of $S$ and $T$ for $T_{444}$. This PD is however not as sparse as the PD in \Cref{tab:param_R49_U,tab:param_R49_V,tab:param_R49_W} and thus can not be used to improve the stability but it may help to understand the set of all PDs of $T_{444}$. We obtain this family in the same way as the non-discretizable two-parameter family in \Cref{tab:non-discr_R23}. The rank of the Jacobian at a generic PD (randomly generated parameters) in this parametrization is 2182. 

\begin{table}[htbp]
\centering
\caption{Matrix $U$ of the 13-parameter family of $\mathrm{PD}_{49}(T_{444})$s with parameters $a$, $b$, $c$, $d$, $f$, $g$, $h$, $n$, $p$, $s$, $t$, $v$, and $w$.}
\label{tab:param_R49_U}
\resizebox{\textwidth}{!}{\begin{minipage}{\textwidth}
\begin{tabular}{c c | c c c c c c c c c c c c c c c c c}
\multirow{50}{*}{\begin{sideways}$U$ \end{sideways}} & & 1 & 2 & 3 & 4 & 5 & 6 & 7 & 8 & 9 & 10 & 11 & 12 & 13 & 14 & 15 & 16 \\
\cmidrule{1-18}
& 1 & $0$ & $0$ & $0$ & $0$ & $1$ & $0$ & $0$ & $0$ & $0$ & $0$ & $0$ & $0$ & $0$ & $0$ & $0$ & $0$  \\
& 2 & $1$ & $0$ & $0$ & $0$ & $p$ & $0$ & $0$ & $0$ & $0$ & $0$ & $0$ & $0$ & $0$ & $0$ & $0$ & $0$ \\  
& 3 & $0$ & $\frac{1}{d}$ & $0$ & $0$ & $0$ & $0$ & $0$ & $0$ & $0$ & $0$ & $1$ & $0$ & $0$ & $0$ & $0$ & $0$ \\ 
& 4 & $0$ & $0$ & $\frac{w}{n}$ & $0$ & $0$ & $0$ & $\frac{1}{n}$ & $\frac{1}{n}$ & $0$ & $0$ & $0$ & $0$ & $0$ & $0$ & $\frac{w}{n}$ & $0$ \\
& 5 & $0$ & $0$ & $0$ & $0$ & $0$ & $0$ & $0$ & $\frac{f}{pw}$ & $\frac{1}{w}$ & $0$ & $0$ & $0$ & $0$ & $1$ & $0$ & $0$ \\
& 6 & $0$ & $0$ & $f$ & $0$ & $0$ & $0$ & $0$ & $0$ & $0$ & $0$ & $0$ & $0$ & $0$ & $0$ & $0$ & $0$ \\
& 7 & $0$ & $0$ & $0$ & $\frac{1}{w}$ & $0$ & $0$ & $0$ & $0$ & $0$ & $\frac{1}{wc}$ & $0$ & $0$ & $\frac{1}{w}$ & $0$ & $fa$ & $0$ \\
& 8 & $0$ & $0$ & $1$ & $0$ & $0$ & $0$ & $0$ & $0$ & $0$ & $0$ & $0$ & $0$ & $0$ & $0$ & $1$ & $0$ \\
& 9 & $0$ & $0$ & $0$ & $0$ & $0$ & $0$ & $0$ & $\frac{f}{p}$ & $1$ & $\frac{-1}{cap}$ & $0$ & $1$ & $0$ & $0$ & $0$ & $0$ \\
& 10 & $0$ & $0$ & $0$ & $0$ & $0$ & $f$ & $f$ & $f$ & $p$ & $\frac{-1}{ac}$ & $\frac{d}{ac}$ & $p$ & $\frac{-1}{a}$ & $0$ & $-fw$ & $0$ \\
& 11 & $0$ & $0$ & $0$ & $0$ & $0$ & $0$ & $0$ & $0$ & $0$ & $0$ & $0$ & $0$ & $1$ & $0$ & $fwa$ & $0$ \\
& 12 & $0$ & $0$ & $0$ & $0$ & $0$ & $1$ & $1$ & $1$ & $0$ & $0$ & $\frac{d}{afc}$ & $0$ & $\frac{-1}{fa}$ & $0$ & $0$ & $0$\\
& 13 & $0$ & $0$ & $0$ & $0$ & $-d$ & $0$ & $0$ & $0$ & $0$ & $\frac{-1}{ap}$ & $0$ & $c$ & $0$ & $0$ & $0$ & $0$ \\
& 14 & $0$ & $\frac{1}{a}$ & $0$ & $\frac{c}{a}$ & $0$ & $0$ & $0$ & $0$ & $0$ & $0$ & $0$ & $0$ & $0$ & $0$ & $0$ & $0$ \\
& 15 & $0$ & $-1$ & $0$ & $-c$ & $0$ & $0$ & $0$ & $0$ & $0$ & $-1$ & $-d$ & $0$ & $0$ & $0$ & $0$ & $0$ \\
& 16 & $0$ & $0$ & $0$ & $0$ & $0$ & $-c$ & $0$ & $0$ & $0$ & $0$ & $\frac{-d}{af}$ & $0$ & $\frac{c}{fa}$ & $0$ & $0$ & $1$ \\
\cmidrule{2-18}
& & 17 & 18 & 19 & 20 & 21 & 22 & 23 & 24 & 25 & 26 & 27 & 28 & 29 & 30 & 31 & 32\\
\cmidrule{2-18}
& 1 & $0$ & $0$ & $0$ & $1$ & $0$ & $0$ & $0$ & $0$ & $0$ & $0$ & $0$ & $\frac{w}{n}$ & $\frac{w}{np}$ & $0$ & $0$ & $\frac{1}{pd}$  \\
& 2 & $0$ & $0$ & $0$ & $p$ & $0$ & $0$ & $0$ & $0$ & $0$ & $0$ & $\frac{1}{n}$ & $\frac{pw}{n}$ & $\frac{w}{n}$ & $0$ & $0$ & $\frac{1}{d}$ \\
& 3 & $0$ & $0$ & $0$ & $0$ & $0$ & $0$ & $0$ & $0$ & $0$ & $0$ & $\frac{-a}{n}$ & $0$ & $0$ & $0$ & $0$ & $0$\\
& 4 & $0$ & $0$ & $\frac{1}{nb}$ & $0$ & $0$ & $0$ & $0$ & $0$ & $0$ & $0$ & $\frac{1}{nf}$ & $0$ & $0$ & $0$ & $0$ & $0$ \\
& 5 & $0$ & $0$ & $0$ & $0$ & $0$ & $0$ & $0$ & $0$ & $0$ & $0$ & $0$ & $1$ & $\frac{1}{p}$ & $0$ & $\frac{-f}{p}$ & $0$ \\
& 6 & $0$ & $f$ & $0$ & $0$ & $0$ & $f$ & $0$ & $0$ & $1$ & $0$ & $\frac{1}{w}$ & $0$ & $1$ & $0$ & $0$ & $\frac{1}{wc}$ \\
& 7 & $0$ & $0$ & $0$ & $0$ & $0$ & $0$ & $1$ & $0$ & $0$ & $\frac{1}{w}$ & $0$ & $0$ & $-a$ & $0$ & $0$ & $0$ \\
& 8 & $0$ & $1$ & $0$ & $0$ & $0$ & $1$ & $0$ & $0$ & $0$ & $0$ & $\frac{1}{wf}$ & $0$ & $0$ & $0$ & $1$ & $\frac{1}{wfc}$ \\
& 9 & $0$ & $0$ & $0$ & $0$ & $0$ & $0$ & $0$ & $0$ & $0$ & $0$ & $0$ & $0$ & $0$ & $0$ & $\frac{-wf}{p}$ & $0$ \\
& 10 & $0$ & $0$ & $0$ & $0$ & $0$ & $0$ & $0$ & $f$ & $0$ & $\frac{-1}{a}$ & $0$ & $0$ & $0$ & $0$ & $-wf$ & $0$ \\
& 11 & $0$ & $0$ & $0$ & $0$ & $0$ & $0$ & $0$ & $0$ & $0$ & $1$ & $0$ & $0$ & $0$ & $0$ & $0$ & $0$\\
& 12  & $0$ & $0$ & $0$ & $0$ & $0$ & $0$ & $0$ & $1$ & $0$ & $0$ & $0$ & $0$ & $0$ & $0$ & $0$  & $0$ \\
& 13 & $1$ & $0$ & $0$ & $0$ & $0$ & $0$ & $0$ & $0$ & $0$ & $0$ & $0$ & $0$ & $0$ & $0$ & $0$ & $\frac{-1}{p}$ \\
& 14 & $0$ & $0$ & $0$ & $0$ & $1$ & $0$ & $0$ & $0$ & $0$ & $0$ & $\frac{-d}{n}$ & $0$ & $-wc$ & $1$ & $0$ & $-1$ \\
& 15 & $0$ & $0$ & $0$ & $0$ & $-a$ & $0$ & $0$ & $0$ & $0$ & $0$ & $\frac{da}{n}$ & $0$ & $wca$ & $0$ & $0$ & $0$  \\
& 16 & $0$ & $0$ & $0$ & $0$ & $0$ & $-wc$ & $0$ & $0$ & $0$ & $0$ & $0$ & $0$ & $0$ & $0$ & $-wc$ & $\frac{-1}{f}$ \\
\cmidrule{2-19}
& & 33 & 34 & 35 & 36 & 37 & 38 & 39 & 40 & 41 & 42 & 43 & 44 & 45 & 46 & 47 & 48 & 49 \\
\cmidrule{2-19}
& 1 & $0$ & $0$ & $0$ & $0$ & $0$ & $0$ & $0$ & $0$ & $0$ & $\frac{1}{n}$ & $0$ & $\frac{t}{pa}$ & $0$ & $0$ & $\frac{-1}{p}$ & $0$ & $0$\\
& 2 & $0$ & $0$ & $0$ & $0$ & $0$ & $0$ & $0$ & $0$ & $0$ & $0$ & $0$ & $0$ & $0$ & $0$ & $0$ & $0$ & $0$ \\
& 3 &  $0$ & $0$ & $0$ & $0$ & $0$ & $1$ & $0$ & $0$ & $1$ & $\frac{pa}{n}$ & $0$ & $t$ & $\frac{1}{n}$ & $0$ & $0$ & $0$ & $\frac{1}{n}$ \\
& 4 & $0$ & $0$ & $1$ & $0$ & $0$ & $0$ & $0$ & $0$ & $0$ & $0$ & $0$ & $1$ & $0$ & $0$ & $\frac{1}{f}$ & $0$ & $0$ \\
& 5 & $\frac{t}{pa}$ & $0$ & $\frac{-tn}{wpa}$ & $0$ & $0$ & $0$ & $0$ & $0$ & $0$ & $\frac{1}{w}$ & $0$ & $0$ & $0$ & $0$ & $0$ & $0$ & $0$ \\
& 6 & $0$ & $0$ & $0$ & $0$ & $0$ & $0$ & $0$ & $0$ & $0$ & $0$ & $0$ & $0$ & $0$ & $0$ & $0$ & $0$ & $0$ \\
& 7 & $t$ & $0$ & $0$ & $\frac{1}{w}$ & $0$ & $0$ & $0$ & $0$ & $0$ & $0$ & $0$ & $0$ & $0$ & $0$ & $0$ & $0$ & $0$ \\
& 8 & $1$ & $0$ & $0$ & $0$ & $0$ & $0$ & $0$ & $0$ & $0$ & $0$ & $0$ & $0$ & $0$ & $0$ & $0$ & $0$ & $0$ \\
& 9 & $0$ & $0$ & $0$ & $\frac{-1}{pa}$ & $0$ & $0$ & $0$ & $\frac{t}{pa}$ & $\frac{-d}{vpa}$ & $0$ & $1$ & $0$ & $0$ & $1$ & $0$ & $\frac{1}{pa}$ & $\frac{1}{ap}$ \\
& 10 & $0$ & $0$ & $\frac{-tn}{a}$ & $\frac{-1}{a}$ & $1$ & $0$ & $1$ & $0$ & $\frac{d}{va}$ & $p$ & $p$ & $0$ & $\frac{1}{a}$ & $-p$ & $n$ & $0$ & $\frac{1}{a}$ \\
& 11 & $0$ & $0$ & $nt$ & $0$ & $0$ & $0$ & $-a$ & $t$ & $\frac{-d}{v}$ & $-ap$ & $0$ & $0$ & $-1$ & $pa$ & $0$ & $1$ & $-1$ \\
& 12 & $0$ & $0$ & $0$ & $0$ & $0$ & $0$ & $0$ & $1$ & $0$ & $0$ & $0$ & $0$ & $0$ & $0$ & $\frac{n}{f}$ & $0$ & $0$ \\
& 13 & $0$ & $\frac{1}{pa}$ & $0$ & $0$ & $0$ & $0$ & $0$ & $0$ & $\frac{d}{pa}$ & $0$ & $0$ & $0$ & $0$ & $0$ & $\frac{d}{p}$ & $0$ & $\frac{c}{ap}$ \\
& 14 & $0$ & $0$ & $0$ & $0$ & $0$ & $0$ & $0$ & $0$ & $0$ & $0$ & $0$ & $0$ & $0$ & $0$ & $0$ & $0$ & $0$ \\
& 15 & $0$ & $1$ & $0$ & $0$ & $0$ & $0$ & $0$ & $0$ & $0$ & $0$ & $0$ & $0$ & $0$ & $0$ & $0$ & $0$ & $0$ \\
& 16 & $0$ & $\frac{1}{t}$ & $0$ & $\frac{c}{t}$ & $0$ & $0$ & $0$ & $0$ & $0$ & $0$ & $0$ & $0$ & $0$ & $0$ & $0$ & $0$ & $0$ \\
\bottomrule
\end{tabular}
\end{minipage}}
\end{table}

\begin{table}[htbp]
\centering
\caption{Matrix $V$ of the 13-parameter family of $\mathrm{PD}_{49}(T_{444})$s with parameters $a$, $b$, $c$, $d$, $f$, $g$, $h$, $n$, $p$, $s$, $t$, $v$, and $w$.}
\label{tab:param_R49_V}
\resizebox{\textwidth}{!}{\begin{minipage}{\textwidth}
\begin{tabular}{c c | c c c c c c c c c c c c c c c c c c}
 & & 1 & 2 & 3 & 4 & 5 & 6 & 7 & 8 & 9 & 10 & 11 & 12 & 13 & 14 & 15 & 16 \\
\cmidrule{1-18}
\multirow{50}{*}{\begin{sideways}$V$ \end{sideways}} & 1 & $0$ & $0$ & $0$ & $0$ & $0$ & $0$ & $0$ & $0$ & $0$ & $0$ & $0$ & $0$ & $0$ & $\frac{n}{wh}$ & $0$ & $0$ \\
& 2 & $0$ & $0$ & $0$ & $w$ & $0$ & $0$ & $0$ & $0$ & $0$ & $0$ & $0$ & $0$ & $0$ & $\frac{-1}{h}$ & $0$ & $0$ \\
& 3 & $0$ & $0$ & $0$ & $0$ & $0$ & $0$ & $0$ & $0$ & $0$ & $0$ & $0$ & $-c$ & $0$ & $0$ & $0$ & $0$ \\
& 4 & $0$ & $0$ & $0$ & $\frac{1}{c}$ & $0$ & $0$ & $0$ & $0$ & $0$ & $1$ & $0$ & $1$ & $0$ & $0$ & $0$ & $0$  \\
& 5 & $0$ & $0$ & $0$ & $0$ & $d$ & $0$ & $-n$ & $0$ & $0$ & $0$ & $0$ & $0$ & $0$ & $0$ & $0$ & $0$ \\
& 6 & $0$ & $0$ & $0$ & $0$ & $0$ & $0$ & $0$ & $0$ & $-w$ & $0$ & $0$ & $0$ & $0$ & $\frac{-1}{b}$ & $0$ & $wc$  \\
& 7 & $0$ & $0$ & $0$ & $0$ & $0$ & $0$ & $1$ & $1$ & $1$ & $0$ & $0$ & $0$ & $0$ & $\frac{1}{bw}$ & $0$ & $0$  \\
& 8 & $0$ & $0$ & $0$ & $0$ & $1$ & $0$ & $0$ & $0$ & $0$ & $0$ & $0$ & $0$ & $0$ & $0$ & $0$ & $1$ \\
& 9 & $1$ & $gd$ & $g$ & $0$ & $0$ & $0$ & $0$ & $0$ & $0$ & $0$ & $0$ & $0$ & $0$ & $0$ & $0$ & $0$  \\
& 10 & $0$ & $0$ & $0$ & $-wh$ & $0$ & $0$ & $0$ & $0$ & $0$ & $0$ & $0$ & $0$ & $0$ & $0$ & $0$ & $0$  \\
& 11 & $0$ & $0$ & $0$ & $0$ & $0$ & $g$ & $b$ & $b$ & $b$ & $0$ & $0$ & $ch$ & $g$ & $\frac{1}{w}$ & $0$ & $-bc$ \\
& 12 & $0$ & $0$ & $0$ & $0$ & $b$ & $0$ & $0$ & $0$ & $0$ & $0$ & $0$ & $0$ & $0$ & $0$ & $0$ & $0$  \\
& 13 & $0$ & $d$ & $1$ & $0$ & $0$ & $0$ & $0$ & $0$ & $0$ & $0$ & $0$ & $0$ & $0$ & $0$ & $0$ & $0$  \\
& 14 & $0$ & $0$ & $\frac{-w}{n}$ & $0$ & $0$ & $0$ & $0$ & $0$ & $0$ & $0$ & $0$ & $0$ & $0$ & $0$ & $1$ & $0$ \\
& 15 & $0$ & $0$ & $0$ & $0$ & $0$ & $1$ & $0$ & $0$ & $0$ & $0$ & $0$ & $0$ & $1$ & $0$ & $0$ & $\frac{-cb}{g}$ \\
& 16 & $0$ & $1$ & $0$ & $0$ & $0$ & $\frac{1}{c}$ & $0$ & $0$ & $0$ & $0$ & $1$ & $0$ & $0$ & $0$ & $0$ & $\frac{-b}{g}$  \\
\cmidrule{2-18}
& & 17 & 18 & 19 & 20 & 21 & 22 & 23 & 24 & 25 & 26 & 27 & 28 & 29 & 30 & 31 & 32\\
\cmidrule{2-18}
& 1 & $0$ & $0$ & $0$ & $\frac{-1}{h}$ & $0$ & $0$ & $0$ & $0$ & $0$ & $0$ & $0$ & $\frac{1}{w}$ & $0$ & $0$ & $0$ & $0$ \\
& 2 & $0$ & $0$ & $0$ & $\frac{w}{nh}$ & $\frac{-gwc}{h}$ & $0$ & $\frac{-1}{h}$ & $0$ & $0$ & $0$ & $0$ & $\frac{-1}{n}$ & $1$ & $0$ & $0$ & $0$ \\
& 3 & $\frac{-bc}{h}$ & $0$ & $0$ & $0$ & $0$ & $0$ & $0$ & $0$ & $0$ & $0$ & $0$ & $0$ & $0$ & $0$ & $0$ & $0$\\
& 4 & $\frac{b}{h}$ & $0$ & $0$ & $0$ & $\frac{-g}{h}$ & $0$ & $\frac{-1}{hcw}$ & $0$ & $0$ & $0$ & $0$ & $0$ & $0$ & $0$ & $0$ & $0$\\
& 5 & $d$ & $0$ & $-n$ & $\frac{1}{b}$ & $0$ & $0$ & $0$ & $\frac{-n}{b}$ & $0$ & $0$ & $0$ & $0$ & $0$ & $0$ & $0$ & $0$\\
& 6 & $0$ & $\frac{1}{b}$ & $0$ & $0$ & $0$ & $1$ & $0$ & $0$ & $0$ & $0$ & $0$ & $0$ & $0$ & $0$ & $1$ & $0$\\
& 7 & $0$ & $0$ & $1$ & $0$ & $0$ & $0$ & $0$ & $\frac{1}{b}$ & $0$ & $0$ & $0$ & $0$ & $0$ & $0$ & $0$ & $0$\\
& 8 & $1$ & $\frac{1}{wcb}$ & $0$ & $\frac{1}{db}$ & $0$ & $\frac{1}{wc}$ & $0$ & $0$ & $0$ & $0$ & $0$ & $0$ & $0$ & $0$ & $0$ & $1$\\
& 9 & $0$ & $\frac{-n}{w}$ & $0$ & $1$ & $gd$ & $0$ & $0$ & $0$ & $0$ & $0$ & $g$ & $0$ & $0$ & $0$ & $0$ & $0$\\
& 10 & $0$ & $1$ & $0$ & $\frac{-w}{n}$ & $gwc$ & $0$ & $0$ & $0$ & $1$ & $0$ & $0$ & $\frac{h}{n}$ & $-h$ & $0$ & $0$ & $0$\\
& 11 & $bc$ & $0$ & $b$ & $0$ & $0$ & $0$ & $\frac{1}{w}$ & $1$ & $0$ & $g$ & $0$ & $0$ & $0$ & $0$ & $0$ & $0$\\
& 12 & $0$ & $\frac{1}{wc}$ & $0$ & $\frac{1}{d}$ & $g$ & $\frac{b}{wc}$ & $0$ & $0$ & $0$ & $0$ & $0$ & $0$ & $0$ & $1$ & $0$ & $b$\\
& 13 & $0$ & $\frac{-n}{wg}$ & $\frac{nb}{g}$ & $0$ & $d$ & $0$ & $0$ & $0$ & $0$ & $0$ & $1$ & $0$ & $0$ & $0$ & $0$ & $0$\\
& 14 & $0$ & $\frac{1}{g}$ & $\frac{-wb}{g}$ & $0$ & $0$ & $0$ & $\frac{-1}{g}$ & $0$ & $0$ & $-w$ & $0$ & $0$ & $0$ & $0$ & $0$ & $0$\\
& 15 & $0$ & $0$ & $0$ & $0$ & $0$ & $0$ & $\frac{1}{gw}$ & $\frac{1}{g}$ & $0$ & $1$ & $0$ & $0$ & $0$ & $0$ & $0$ & $0$\\
& 16 & $0$ & $0$ & $0$ & $0$ & $1$ & $0$ & $0$ & $\frac{1}{gc}$ & $0$ & $0$ & $0$ & $0$ & $0$ & $0$ & $0$ & $0$\\
\cmidrule{2-19}
& &  33 & 34 & 35 & 36 & 37 & 38 & 39 & 40 & 41 & 42 & 43 & 44 & 45 & 46 & 47 & 48 & 49 \\
\cmidrule{2-19}
& 1 & $0$ & $0$ & $ns$ & $0$ & $0$ & $\frac{g}{h}$ & $\frac{-gn}{h}$ & $0$ & $0$ & $1$ & $0$ & $s$ & $n$ & $0$ & $0$ & $0$ & $0$\\
& 2 & $s$ & $0$ & $0$ & $0$ & $0$ & $0$ & $0$ & $0$ & $0$ & $0$ & $0$ & $0$ & $0$ & $0$ & $0$ & $0$ & $0$\\
& 3 & $0$ & $0$ & $0$ & $0$ & $0$ & $\frac{g}{hn}$ & $\frac{-g}{h}$ & $s$ & $0$ & $0$ & $\frac{-b}{h}$ & $0$ & $1$ & $0$ & $0$ & $0$ & $1$\\
& 4 & $0$ & $s$ & $0$ & $\frac{gs}{bc}$ & $0$ & $0$ & $0$ & $0$ & $0$ & $0$ & $\frac{b}{ch}$ & $0$ & $0$ & $0$ & $0$ & $0$ & $0$\\
& 5 & $0$ & $0$ & $0$ & $0$ & $0$ & $0$ & $0$ & $0$ & $\frac{dg}{b}$ & $0$ & $0$ & $1$ & $0$ & $\frac{dg}{b}$ & $1$ & $0$ & $0$\\
& 6 & $1$ & $0$ & $0$ & $\frac{-gw}{b}$ & $0$ & $0$ & $0$ & $0$ & $0$ & $0$ & $-w$ & $0$ & $0$ & $0$ & $0$ & $0$ & $0$\\
& 7 & $0$ & $0$ & $1$ & $0$ & $0$ & $0$ & $0$ & $1$ & $0$ & $0$ & $1$ & $0$ & $0$ & $\frac{-vg}{b}$ & $0$ & $1$ & $0$\\
& 8 & $0$ & $1$ & $0$ & $0$ & $0$ & $0$ & $0$ & $0$ & $0$ & $0$ & $0$ & $0$ & $0$ & $0$ & $0$ & $0$ & $0$\\
& 9 & $0$ & $0$ & $0$ & $0$ & $0$ & $0$ & $0$ & $0$ & $0$ & $0$ & $0$ & $0$ & $0$ & $0$ & $0$ & $0$ & $0$\\
& 10 & $0$ & $0$ & $0$ & $0$ & $0$ & $0$ & $0$ & $0$ & $0$ & $0$ & $0$ & $0$ & $0$ & $0$ & $0$ & $0$ & $0$\\
& 11 & $0$ & $0$ & $b$ & $g$ & $1$ & $\frac{-g}{n}$ & $g$ & $0$ & $0$ & $0$ & $b$ & $0$ & $-h$ & $-vg$ & $0$ & $0$ & $-h$\\
&12 & $0$ & $0$ & $0$ & $0$ & $0$ & $0$ & $0$ & $0$ & $0$ & $0$ & $0$ & $0$ & $0$ & $0$ & $0$ & $0$ & $0$\\
& 13 & $0$ & $0$ & $0$ & $0$ & $0$ & $1$ & $0$ & $0$ & $0$ & $0$ & $0$ & $\frac{-b}{g}$ & $0$ & $0$ & $0$ & $0$ & $0$\\
&14 & $\frac{-b}{g}$ & $0$ & $\frac{-bw}{g}$ & $0$ & $0$ & $0$ & $-w$ & $0$ & $0$ & $0$ & $0$ & $0$ & $0$ & $0$ & $0$ & $0$ & $0$\\
& 15 & $0$ & $0$ & $0$ & $1$ & $0$ & $0$ & $1$ & $\frac{-b}{g}$ & $0$ & $0$ & $0$ & $0$ & $0$ & $v$ & $0$ & $\frac{-b}{g}$ & $0$\\
& 16 & $0$ & $\frac{-b}{g}$ & $0$ & $0$ & $0$ & $\frac{1}{d}$ & $0$ & $0$ & $1$ & $0$ & $0$ & $0$ & $0$ & $1$ & $0$ & $0$ & $0$\\
\bottomrule
\end{tabular}
\end{minipage}}
\end{table}

\begin{table}[htbp]
\caption{Matrix $W$ of the 13-parameter family of $\mathrm{PD}_{49}(T_{444})$s with parameters $a$, $b$, $c$, $d$, $f$, $g$, $h$, $n$, $p$, $s$, $t$, $v$, $w$ and where $a_1 = \frac{t(b-hs)+fab}{bpsa}$ and $a_2 = \frac{fab-t(b+hs)}{sg}$.}
\label{tab:param_R49_W}
\resizebox{\textwidth}{!}{\begin{minipage}{\textwidth}
\begin{tabular}{ c | c c c c c c c c c c c c c c c c c c}
 & 1 & 2 & 3 & 4 & 5 & 6 & 7 & 8 & 9 & 10 & 11 & 12 & 13 & 14 & 15 & 16 \\
\cmidrule{1-17}
1 & $0$ & $0$ & $0$ & $0$ & $0$ & $0$ & $0$ & $0$ & $\frac{h}{b}$ & $0$ & $0$ & $\frac{p}{c}$ & $0$ & $-h$ & $0$ & $0$\\
2 & $0$ & $0$ & $0$ & $0$ & $0$ & $0$ & $\frac{-p}{f}$ & $\frac{p}{f}$ & $-1$ & $0$ & $0$ & $\frac{-bp}{ch}$ & $0$ & $0$ & $0$ & $0$\\
 3 & $-p$ & $0$ & $0$ & $0$ & $\frac{-1}{db}$ & $0$ & $0$ & $0$ & $\frac{1}{b}$ & $0$ & $0$ & $\frac{p}{ch}$ & $0$ & $-1$ & $0$ & $0$\\
4 & $0$ & $0$ & $0$ & $0$ & $0$ & $0$ & $\frac{-pg}{bf}$ & $\frac{pg}{bf}$ & $0$ & $0$ & $0$ & $0$ & $0$ & $0$ & $0$ & $0$\\
5 & $0$ & $0$ & $0$ & $a$ & $0$ & $0$ & $0$ & $0$ & $0$ & $-a$ & $0$ & $\frac{-1}{c}$ & $0$ & $0$ & $0$ & $0$\\
6 & $0$ & $0$ & $0$ & $0$ & $\frac{1}{pd}$ & $0$ & $\frac{1}{f}$ & $\frac{-1}{f}$ & $\frac{1}{p}$ & $0$ & $0$ & $0$ & $0$ & $0$ & $0$ & $0$\\
7 & $1$ & $0$ & $0$ & $0$ & $0$ & $0$ & $0$ & $0$ & $0$ & $0$ & $0$ & $0$ & $0$ & $0$ & $0$ & $0$\\
8 & $-g$ & $a$ & $\frac{1}{wf}$ & $0$ & $0$ & $\frac{-1}{f}$ & $0$ & $0$ & $0$ & $0$ & $\frac{-a}{d}$ & $0$ & $-a$ & $0$ & $\frac{-1}{f}$ & $0$\\
9 & $0$ & $0$ & $0$ & $0$ & $0$ & $0$ & $0$ & $0$ & $0$ & $-1$ & $0$ & $\frac{-1}{ca}$ & $0$ & $0$ & $0$ & $0$\\
10 & $0$ & $0$ & $0$ & $0$ & $0$ & $\frac{-b}{agf}$ & $0$ & $0$ & $0$ & $s$ & $\frac{-b}{dg}$ & $\frac{b}{ach}$ & $\frac{-b}{g}$ & $0$ & $0$ & $\frac{-1}{t}$\\
11 & $0$ & $\frac{1}{g}$ & $0$ & $\frac{-1}{h}$ & $0$ & $0$ & $0$ & $0$ & $0$ & $0$ & $0$ & $0$ & $0$ & $0$ & $0$ & $0$\\ 
12 & $0$ & $0$ & $0$ & $0$ & $0$ & $\frac{-1}{fa}$ & $0$ & $0$ & $0$ & $0$ & $\frac{-1}{d}$ & $0$ & $-1$ & $0$ & $0$ & $0$\\
13 & $0$ & $0$ & $0$ & $0$ & $0$ & $0$ & $0$ & $\frac{-1}{s}$ & $a_1$ & $0$ & $0$ & $0$ & $\frac{fab}{sg}$ & $\frac{th}{pa}$ & $\frac{b}{sg}$ & $0$\\
14 & $0$ & $0$ & $0$ & $0$ & $0$ & $\frac{b}{g}$ & $-1$ & $1$ & $\frac{-f}{p}$ & $0$ & $0$ & $0$ & $0$ & $0$ & $0$ & $1$\\
15 & $0$ & $0$ & $\frac{-1}{gw}$ & $0$ & $0$ & $\frac{-1}{g}$ & $\frac{1}{b}$ & $0$ & $0$ & $0$ & $0$ & $0$ & $0$ & $0$ & $0$ & $\frac{-1}{b}$\\
 16 & $0$ & $0$ & $0$ & $0$ & $0$ & $1$ & $\frac{-g}{b}$ & $0$ & $0$ & $0$ & $0$ & $0$ & $fa$ & $0$ & $1$ & $0$\\
\cmidrule{1-17}
 & 17 & 18 & 19 & 20 & 21 & 22 & 23 & 24 & 25 & 26 & 27 & 28 & 29 & 30 & 31 & 32\\
\cmidrule{1-17}
 1 & $0$ & $0$ & $0$ & $0$ & $0$ & $0$ & $0$ & $0$ & $0$ & $0$ & $0$ & $0$ & $0$ & $0$ & $0$ & $0$\\
 2 & $1$ & $0$ & $0$ & $0$ & $0$ & $0$ & $0$ & $0$ & $0$ & $0$ & $0$ & $0$ & $0$ & $0$ & $0$ & $0$\\
 3 & $\frac{-1}{b}$ & $0$ & $0$ & $1$ & $0$ & $0$ & $0$ & $0$ & $0$ & $0$ & $0$ & $\frac{n}{h}$ & $0$ & $0$ & $0$ & $0$\\
4 & $0$ & $0$ & $\frac{pag}{t}$ & $0$ & $0$ & $0$ & $0$ & $0$ & $0$ & $0$ & $0$ & $0$ & $0$ & $0$ & $0$ & $0$\\
5 & $0$ & $0$ & $0$ & $0$ & $0$ & $0$ & $0$ & $0$ & $h$ & $0$ & $0$ & $\frac{n}{p}$ & $-1$ & $0$ & $0$ & $0$\\
6 & $0$ & $0$ & $0$ & $0$ & $0$ & $\frac{1}{f}$ & $0$ & $0$ & $0$ & $0$ & $0$ & $0$ & $0$ & $\frac{-b}{v}$ & $\frac{-1}{f}$ & $-1$\\
7 & $0$ & $0$ & $0$ & $0$ & $0$ & $0$ & $0$ & $0$ & $1$ & $0$ & $0$ & $0$ & $0$ & $\frac{1}{v}$ & $0$ & $0$\\
8 & $0$ & $0$ & $0$ & $0$ & $0$ & $0$ & $0$ & $0$ & $0$ & $-a$ & $n$ & $0$ & $0$ & $0$ & $0$ & $0$\\
9 & $0$ & $0$ & $0$ & $0$ & $0$ & $0$ & $-h$ & $0$ & $0$ & $\frac{h}{g}$ & $0$ & $0$ & $0$ & $0$ & $0$ & $0$\\
10 & $\frac{-1}{pa}$ & $0$ & $0$ & $0$ & $0$ & $0$ & $0$ & $0$ & $0$ & $0$ & $0$ & $0$ & $0$ & $0$ & $0$ & $0$\\
11 & $0$ & $0$ & $0$ & $0$ & $\frac{1}{ga}$ & $0$ & $-1$ & $0$ & $0$ & $\frac{1}{g}$ & $0$ & $0$ & $0$ & $\frac{1}{va}$ & $0$ & $0$\\
12 & $0$ & $0$ & $0$ & $0$ & $0$ & $0$ & $0$ & $0$ & $0$ & $-1$ & $0$ & $0$ & $0$ & $0$ & $0$ & $0$\\
13 & $0$ & $0$ & $0$ & $0$ & $0$ & $0$ & $th$ & $0$ & $0$ & $a_2$ & $0$ & $0$ & $0$ & $0$ & $\frac{-1}{s}$ & $0$\\
14 & $0$ & $0$ & $0$ & $0$ & $0$ & $0$ & $0$ & $b$ & $0$ & $0$ & $0$ & $0$ & $0$ & $0$ & $1$ & $0$\\
15 & $0$ & $1$ & $1$ & $0$ & $0$ & $\frac{-1}{b}$ & $0$ & $-1$ & $-f$ & $0$ & $0$ & $0$ & $0$ & $0$ & $0$ & $0$\\
 16 & $0$ & $0$ & $-g$ & $0$ & $0$ & $0$ & $0$ & $g$ & $0$ & $fa$ & $0$ & $0$ & $0$ & $0$ & $0$ & $0$\\
\cmidrule{1-18}
 &  33 & 34 & 35 & 36 & 37 & 38 & 39 & 40 & 41 & 42 & 43 & 44 & 45 & 46 & 47 & 48 & 49 \\
\cmidrule{1-18}
1 & $0$ & $0$ & $0$ & $0$ & $hp$ & $0$ & $0$ & $0$ & $0$ & $n$ & $\frac{-h}{b}$ & $0$ & $-ap$ & $0$ & $0$ & $0$ & $ap$\\
 2 & $0$ & $0$ & $0$ & $0$ & $-bp$ & $0$ & $0$ & $0$ & $0$ & $0$ & $1$ & $0$ & $0$ & $0$ & $-p$ & $0$ & $0$\\
3 & $0$ & $0$ & $0$ & $0$ & $p$ & $0$ & $0$ & $0$ & $0$ & $0$ & $\frac{-1}{b}$ & $0$ & $0$ & $0$ & $0$ & $0$ & $0$\\
4 & $0$ & $0$ & $\frac{pag}{bt}$ & $0$ & $0$ & $pa$ & $p$ & $0$ & $\frac{pa}{d}$ & $\frac{sgn}{b}$ & $0$ & $\frac{-ag}{bt}$ & $\frac{pga}{h}$ & ${1}$ & $\frac{-pg}{b}$ & $0$ & $0$\\
 5 & $0$ & $0$ & $0$ & $0$ & $-h$ & $0$ & $0$ & $0$ & $0$ & $\frac{-n}{p}$ & $0$ & $0$ & $a$ & $0$ & $0$ & $0$ & $-a$\\
 6 & $0$ & $0$ & $0$ & $0$ & $b$ & $0$ & $0$ & $0$ & $0$ & $0$ & $0$ & $0$ & $0$ & $0$ & $1$ & $0$ & $0$\\
7 & $0$ & $0$ & $0$ & $0$ & $-1$ & $0$ & $0$ & $0$ & $0$ & $0$ & $0$ & $0$ & $0$ & $0$ & $0$ & $0$ & $0$\\
8 & $0$ & $0$ & $0$ & $0$ & $g$ & $0$ & $0$ & $0$ & $0$ & $0$ & $0$ & $0$ & $0$ & $0$ & $0$ & $0$ & $0$\\
9 & $0$ & $0$ & $0$ & $0$ & $\frac{-h}{a}$ & $0$ & $\frac{h}{ga}$ & $0$ & $0$ & $0$ & $0$ & $0$ & $1$ & $0$ & $0$ & $0$ & $-1$\\
10 & $0$ & $1$ & $0$ & $\frac{-b}{g}$ & $\frac{2b}{a}$ & $0$ & $0$ & $0$ & $\frac{b}{dg}$ & $0$ & $\frac{-1}{ap}$ & $0$ & $0$ & $\frac{b}{gpa}$ & $0$ & $2$ & $0$\\
 11 & $0$ & $0$ & $0$ & $0$ & $\frac{-1}{a}$ & $\frac{1}{g}$ & $\frac{1}{ga}$ & $0$ & $0$ & $0$ & $0$ & $0$ & $\frac{1}{h}$ & $0$ & $0$ & $0$ & $0$\\
12 & $0$ & $0$ & $0$ & $0$ & $\frac{g}{a}$ & $-1$ & $\frac{-1}{a}$ & $0$ & $0$ & $0$ & $0$ & $0$ & $\frac{-g}{h}$ & $0$ & $0$ & $0$ & $0$\\
13 & $\frac{1}{s}$ & $0$ & $\frac{1}{s}$ & $\frac{tb}{gs}$ & $0$ & $0$ & $\frac{-th}{ga}$ & $\frac{1}{s}$ & $0$ & $0$ & $\frac{th}{bpa}$ & $0$ & $0$ & $0$ & $0$ & $\frac{-2t}{s}$ & $0$\\
14 & $0$ & $0$ & $0$ & $0$ & $-bf$ & $0$ & $0$ & $0$ & $0$ & $0$ & $0$ & $0$ & $0$ & $0$ & $0$ & $0$ & $0$\\
15 & $0$ & $0$ & $0$ & $0$ & $f$ & $0$ & $0$ & $0$ & $0$ & $0$ & $0$ & $0$ & $0$ & $0$ & $0$ & $0$ & $0$\\
16 & $0$ & $0$ & $0$ & $0$ & $-gf$ & $0$ & $0$ & $0$ & $0$ & $0$ & $0$ & $0$ & $0$ & $0$ & $0$ & $0$ & $0$\\
\bottomrule
\end{tabular}
\end{minipage}}
\end{table}

\begin{table}[htbp]
\centering
\caption{Submatrix $A$ of the CS five-parameter family of $\mathrm{PD}_{49}(T_{444})$s with $S=13$ and $T=12$, and parameters $a$, $b$, $d$, $g$, and $h$, and where $c = \frac{b^2(1+ag)}{1+abg}$.}
\label{tab:5param_R49_S13_A}
\begin{tabular}{c c | c c c c c c c c c c c c c c c c c}
\multirow{17}{*}{\begin{sideways}$A$ \end{sideways}} & & 1 & 2 & 3 & 4 & 5 & 6 & 7 & 8 & 9 & 10 & 11 & 12 & 13 \\
 \cmidrule{1-15}
 & 1 & $0$ & $0$ & $0$ & $1$ & $0$ & $-1$ & $0$ & $0$ & $1$ & $0$ & $0$ & $0$ & $0$ \\
 & 2 & $0$ & $0$ & $0$ & $\frac{a}{b}$ & $0$ & $\frac{-a}{b}$ & $0$ & $0$ & $0$ & $\frac{-a}{c}$ & $0$ & $\frac{a}{c}$ & $0$\\
 & 3 & $0$ & $0$ & $0$ & $\frac{1}{b}$ & $0$ & $\frac{-1}{b}$ & $0$ & $0$ & $0$ & $0$ & $0$ & $0$ & $0$\\
 & 4 & $0$ & $\frac{-a}{cd}$ & $0$ & $0$ & $0$ & $0$ & $0$ & $0$ & $0$ & $0$ & $0$ & $0$ & $\frac{a}{cd}$\\
 & 5 & $0$ & $0$ & $\frac{-c}{a}$ & $0$ & $0$ & $0$ & $0$ & $0$ & $0$ & $\frac{c}{a}$ & $0$ & $0$ & $0$\\
 & 6 & $0$ & $0$ & $0$ & $0$ & $0$ & $0$ & $0$ & $0$ & $0$ & $0$ & $0$ & $0$ & $1$\\
 & 7 & $\frac{1}{a}$ & $0$ & $\frac{-1}{a}$ & $0$ & $\frac{-1}{a}$ & $0$ & $0$ & $0$ & $0$ & $\frac{1}{a}$ & $\frac{1}{a}$ & $\frac{-1}{a}$ & $0$\\
 & 8 & $\frac{1}{d}$ & $0$ & $0$ & $0$ & $\frac{-1}{d}$ & $0$ & $0$ & $0$ & $0$ & $0$ & $\frac{1}{d}$ & $\frac{-1}{d}$ & $0$\\
 & 9 & $0$ & $0$ & $c$ & $0$ & $0$ & $b$ & $0$ & $0$ & $-b$ & $-c$ & $0$ & $0$ & $0$\\
 & 10 & $-a$ & $0$ & $0$ & $0$ & $0$ & $0$ & $a$ & $0$ & $0$ & $0$ & $0$ & $0$ & $0$\\
 & 11 & $-1$ & $0$ & $1$ & $0$ & $1$ & $0$ & $1$ & $0$ & $0$ & $-1$ & $-1$ & $1$ & $0$\\
 & 12 & $0$ & $\frac{a}{d}$ & $0$ & $0$ & $\frac{a}{d}$ & $0$ & $0$ & $0$ & $0$ & $0$ & $\frac{-a}{d}$ & $0$ & $\frac{-a}{d}$\\
 & 13 & $0$ & $\frac{cd}{a}$ & $\frac{-cd}{a}$ & $0$ & $0$ & $0$ & $0$ & $\frac{-cd}{a}$ & $0$ & $\frac{cd}{a}$ & $\frac{cd}{a}$ & $\frac{-cd}{a}$ & $0$\\
 & 14 & $0$ & $0$ & $0$ & $0$ & $0$ & $0$ & $0$ & $0$ & $0$ & $0$ & $0$ & $0$ & $0$\\
 & 15 & $0$ & $0$ & $\frac{-d}{a}$ & $0$ & $0$ & $0$ & $0$ & $0$ & $0$ & $\frac{d}{a}$ & $\frac{d}{a}$ & $\frac{-d}{a}$ & $0$\\
 & 16 & $0$ & $-1$ & $0$ & $0$ & $0$ & $0$ & $0$ & $1$ & $0$ & $0$ & $0$ & $0$ & $1$\\
\bottomrule
\end{tabular}
\end{table}

\begin{table}[htbp]
\centering
\caption{Submatrices $B$, $C$, and $D$ of the CS five-parameter family of $\mathrm{PD}_{49}(T_{444})$s with $S=13$ and $T=12$, and parameters $a$, $b$, $d$, $g$, and $h$, and where $c = \frac{b^2(1+ag)}{1+abg}$.}
\label{tab:5param_R49_S13_BCD}
\resizebox{0.9 \textwidth}{!}{\begin{minipage}{\textwidth}
\begin{tabular}{c c | c c c c c c c c c c c c c}
\multirow{17}{*}{\begin{sideways}$B$ \end{sideways}} & & 1 & 2 & 3 & 4 & 5 & 6 & 7 & 8 & 9 & 10 & 11 & 12 \\
 \cmidrule{1-14}
  & 1 & $0$ & $0$ & $0$ & $0$ & $0$ & $0$ & $0$ & $0$ & $0$ & $0$ & $1$ & $0$\\
 & 2 & $\frac{-ah}{cd}$ & $0$ & $0$ & $0$ & $0$ & $0$ & $1$ & $1$ & $\frac{-d}{c}$ & $0$ & $\frac{-1}{gb^2}$ & $0$\\
 & 3 & $0$ & $0$ & $0$ & $0$ & $0$ & $0$ & $0$ & $0$ & $0$ & $0$ & $\frac{1}{b}$ & $0$\\
 & 4 & $\frac{a}{cd}$ & $0$ & $0$ & $0$ & $0$ & $0$ & $0$ & $0$ & $0$ & $0$ & $0$ & $0$\\
 & 5 & $0$ & $g$ & $0$ & $0$ & $0$ & $0$ & $0$ & $0$ & $0$ & $0$ & $0$ & $0$\\
 & 6 & $0$ & $0$ & $0$ & $-h$ & $0$ & $0$ & $0$ & $0$ & $0$ & $0$ & $0$ & $a$\\
 & 7 & $0$ & $0$ & $0$ & $0$ & $0$ & $1$ & $0$ & $0$ & $0$ & $0$ & $0$ & $1$\\
 & 8 & $0$ & $0$ & $0$ & $1$ & $0$ & $\frac{a}{d}$ & $0$ & $0$ & $0$ & $0$ & $0$ & $0$\\
 & 9 & $0$ & $1$ & $0$ & $0$ & $0$ & $0$ & $0$ & $0$ & $0$ & $0$ & $-b$ & $0$\\
 & 10 & $\frac{ah}{d}$ & $0$ & $0$ & $0$ & $0$ & $0$ & $0$ & $-c$ & $0$ & $0$ & $\frac{1}{g}$ & $-a^2$\\
 & 11 & $0$ & $0$ & $0$ & $0$ & $0$ & $0$ & $0$ & $0$ & $0$ & $0$ & $-1$ & $-a$\\
 & 12 & $\frac{-a}{d}$ & $0$ & $0$ & $0$ & $0$ & $0$ & $0$ & $0$ & $1$ & $0$ & $0$ & $0$\\
 & 13 & $0$ & $0$ & $1$ & $0$ & $c$ & $0$ & $0$ & $\frac{c^2dg}{a}$ & $0$ & $-gc$ & $0$ & $0$\\
 & 14 & $-h$ & $0$ & $0$ & $0$ & $0$ & $0$ & $0$ & $0$ & $0$ & $1$ & $0$ & $ad$\\
 & 15 & $0$ & $0$ & $0$ & $0$ & $1$ & $0$ & $0$ & $\frac{cdg}{a}$ & $0$ & $-g$ & $0$ & $d$\\
 & 16 & $1$ & $0$ & $0$ & $0$ & $0$ & $0$ & $0$ & $0$ & $0$ & $0$ & $0$ & $0$\\
 \cmidrule{1-14}
 \multirow{17}{*}{\begin{sideways}$C$ \end{sideways}} & & 1 & 2 & 3 & 4 & 5 & 6 & 7 & 8 & 9 & 10 & 11 & 12 \\
 \cmidrule{1-14}
  & 1 & $0$ & $0$ & $1$ & $0$ & $0$ & $0$ & $0$ & $0$ & $0$ & $0$ & $0$ & $\frac{-1}{d}$\\
 & 2 & $0$ & $0$ & $0$ & $0$ & $0$ & $0$ & $0$ & $0$ & $0$ & $0$ & $0$ & $0$\\
& 3 & $0$ & $\frac{-1}{gc}$ & $0$ & $0$ & $0$ & $0$ & $0$ & $0$ & $0$ & $0$ & $0$ & $\frac{-1}{cd}$\\
& 4 & $0$ & $\frac{-a}{cdg}$ & $\frac{-a}{cd}$ & $\frac{a}{cdh}$ & $0$ & $0$ & $0$ & $0$ & $0$ & $0$ & $0$ & $0$\\
& 5 & $0$ & $0$ & $0$ & $0$ & $\frac{-cdg}{a^2}$ & $0$ & $\frac{1}{a}$ & $-1$ & $0$ & $1$ & $0$ & $0$\\
 & 6 & $\frac{-1}{h}$ & $0$ & $0$ & $0$ & $0$ & $0$ & $\frac{1}{c}$ & $0$ & $\frac{1}{d}$ & $0$ & $0$ & $0$\\
 & 7 & $0$ & $\frac{-1}{ag}$ & $0$ & $0$ & $0$ & $0$ & $\frac{1}{ac}$ & $0$ & $\frac{1}{ad}$ & $0$ & $0$ & $0$\\
 & 8 & $0$ & $\frac{-1}{gd}$ & $0$ & $0$ & $\frac{g}{a}$ & $0$ & $0$ & $\frac{a}{cd}$ & $0$ & $\frac{-a}{cd}$ & $0$ & $0$\\
 & 9 & $0$ & $0$ & $-c$ & $0$ & $\frac{-cd}{a^2}$ & $0$ & $-1$ & $\frac{-1}{g}$ & $0$ & $0$ & $-1$ & $\frac{c}{d}$\\
 & 10 & $0$ & $0$ & $0$ & $0$ & $0$ & $\frac{1}{a}$ & $\frac{-a}{c}$ & $0$ & $\frac{-a}{d}$ & $0$ & $0$ & $0$\\
 & 11 & $0$ & $\frac{1}{g}$ & $0$ & $0$ & $0$ & $0$ & $\frac{-1}{c}$ & $0$ & $\frac{-1}{d}$ & $0$ & $0$ & $\frac{1}{d}$\\
 & 12 & $0$ & $\frac{a}{dg}$ & $\frac{a}{d}$ & $\frac{-a}{dh}$ & $\frac{1}{a}$ & $\frac{-1}{ad}$ & $0$ & $0$ & $0$ & $0$ & $0$ & $0$\\
 & 13 & $0$ & $0$ & $\frac{cd}{a}$ & $0$ & $0$ & $0$ & $\frac{d}{a}$ & $0$ & $\frac{c}{a}$ & $0$ & $0$ & $\frac{-c}{a}$\\
 & 14 & $0$ & $0$ & $0$ & $0$ & $0$ & $0$ & $\frac{d}{c}$ & $0$ & $1$ & $0$ & $0$ & $0$\\
 & 15 & $0$ & $\frac{-d}{ag}$ & $0$ & $0$ & $0$ & $0$ & $\frac{d}{ac}$ & $0$ & $\frac{1}{a}$ & $0$ & $0$ & $\frac{-1}{a}$\\
& 16 & $0$ & $\frac{-1}{g}$ & $-1$ & $\frac{1}{h}$ & $0$ & $0$ & $0$ & $0$ & $0$ & $0$ & $0$ & $0$\\
 \cmidrule{1-14}
 \multirow{17}{*}{\begin{sideways}$D$ \end{sideways}} & & 1 & 2 & 3 & 4 & 5 & 6 & 7 & 8 & 9 & 10 & 11 & 12 \\
 \cmidrule{1-14}
  & 1 & $0$ & $0$ & $0$ & $0$ & $0$ & $0$ & $a$ & $0$ & $a$ & $0$ & $0$ & $0$\\
 & 2 & $0$ & $1$ & $0$ & $\frac{ah}{cd}$ & $0$ & $0$ & $0$ & $\frac{a}{c}$ & $0$ & $-d$ & $a$ & $0$\\
 & 3 & $0$ & $0$ & $0$ & $0$ & $0$ & $0$ & $\frac{a}{c}$ & $0$ & $\frac{a}{c}$ & $0$ & $1$ & $0$\\
 & 4 & $\frac{ah}{cd}$ & $0$ & $1$ & $0$ & $0$ & $0$ & $0$ & $\frac{-a}{cd}$ & $0$ & $1$ & $0$ & $0$\\
 & 5 & $0$ & $0$ & $0$ & $0$ & $0$ & $0$ & $c$ & $0$ & $0$ & $0$ & $0$ & $0$\\
 & 6 & $0$ & $0$ & $0$ & $1$ & $0$ & $a$ & $0$ & $0$ & $0$ & $0$ & $0$ & $0$\\
 & 7 & $0$ & $0$ & $0$ & $0$ & $1$ & $1$ & $1$ & $\frac{-1}{a}$ & $1$ & $0$ & $0$ & $0$\\
 & 8 & $1$ & $0$ & $0$ & $0$ & $\frac{a}{d}$ & $0$ & $0$ & $\frac{-1}{d}$ & $\frac{a}{d}$ & $0$ & $0$ & $0$\\
 & 9 & $0$ & $0$ & $0$ & $0$ & $0$ & $0$ & $-ac$ & $0$ & $-ac$ & $0$ & $0$ & $0$\\
 & 10 & $0$ & $-c$ & $0$ & $\frac{-ah}{d}$ & $0$ & $-a^2$ & $0$ & $0$ & $0$ & $0$ & $0$ & $0$\\
 & 11 & $0$ & $0$ & $0$ & $0$ & $-a$ & $-a$ & $-a$ & $1$ & $-a$ & $0$ & $0$ & $0$\\
 & 12 & $\frac{-ah}{d}$ & $0$ & $-c$ & $0$ & $\frac{-a^2}{d}$ & $0$ & $0$ & $\frac{a}{d}$ & $0$ & $0$ & $0$ & $1$\\
 & 13 & $0$ & $0$ & $0$ & $0$ & $0$ & $0$ & $cd$ & $0$ & $cd$ & $0$ & $0$ & $0$\\
 & 14 & $0$ & $0$ & $0$ & $h$ & $0$ & $0$ & $0$ & $0$ & $0$ & $0$ & $0$ & $0$\\
 & 15 & $0$ & $0$ & $0$ & $0$ & $d$ & $0$ & $d$ & $\frac{-d}{a}$ & $d$ & $0$ & $0$ & $0$\\
 & 16 & $h$ & $0$ & $0$ & $0$ & $a$ & $0$ & $0$ & $-1$ & $0$ & $0$ & $0$ & $0$\\
\bottomrule
\end{tabular}
\end{minipage}}
\end{table}

\subsection{Approximate Decompositions} \label{sec:approx_decomp}

In this subsection, we give the cost function and corresponding values of $u$ and $l$ of the most accurate numerical PDs of lower rank we were able to obtain. For these experiments we use many (thousand) random starting points (close to zero, e.g., of size $10^{-2}$) and set $u:=-l:=1$. We do not use any equality constraints and let the method run for a fixed number of iterations, e.g., 1000. Afterwards, we use the most accurate approximate PD that is obtained from these experiments as starting point and gradually increase $u$ and $-l$. The results are summarized in \Cref{tab:results_approx} for different values of $u$ and $l$. In this table, $\tilde{x}$ denotes the approximate PD and $u_{\max}$ and $l_{\max}$ are the highest values of the bounds (in absolute value) that we are able to obtain. If we increase the bounds further, the AL method is not able to converge to a more accurate PD. 

\begin{table}[htbp]
\centering
\caption{Overview of the most accurate approximate PDs $\tilde{x}$ for different values of $u$ and $l$.}
\label{tab:results_approx}
\begin{tabular}{c c c | c c c c}
\multirow{2}{*}{\mbox{ \boldmath $m$}} & \multirow{2}{*}{\mbox{ \boldmath $p$}} & \multirow{2}{*}{\mbox{ \boldmath $n$}} & \multirow{2}{*}{\mbox{ \boldmath $R$}} & \multirow{2}{*}{\mbox{ \boldmath $u_{\max} = -l_{\max}$}} & \multicolumn{2}{c}{ \mbox{ \boldmath $f(\tilde{x})$}} \\
 &  & & & & \mbox{ \boldmath $u:=-l:=1$} & \mbox{ \boldmath $u_{\max}$, $l_{\max}$} \\
\midrule
3 & 3 & 3 & 22 & $168$  & $5\cdot 10^{-3}$  & $2 \cdot 10^{-15}$ \\
2 & 4 & 5 & 32 & $29$ & $6 \cdot 10^{-4}$ & $2\cdot 10^{-19}$ \\
4 & 4 & 4 & 48 & $137$ & $5\cdot 10^{-3}$  & $2 \cdot 10^{-14}$ \\
\end{tabular}
\end{table}

For $T_{333}$, we tried to lower the rank to 22 using more than thousand random starting points, described in \Cref{sec:starting_points}. However, when we bound the elements to the interval $[-1,1]$, the most accurate PD we are able to obtain has a cost function $f(\tilde{x})$ of $5\cdot 10^{-3}$. When we make the interval gradually larger till $u:=-l:=168$, the most accurate PD we are able to obtain has a cost function of approximately $2\cdot 10^{-15}$. 

Then, for $T_{245}$, rank 32, the most accurate numerical PD has a cost function of approximately $6 \cdot 10^{-4}$ for $u:=-l:=1$, and we can increase $u$ and $-l$ to $29$ to obtain a cost function of approximately $2\cdot 10^{-19}$. 

Lastly, for rank 48 of $T_{444}$, the most accurate numerical PD has a cost function of approximately $5 \cdot 10^{-3}$ for $u:=-l:=1$ and $2\cdot 10^{-14}$ for $u:=-l:=137$. 

\section{Conclusion}\label{sec13}

Using the AL method with specific equality and inequality constraints, we are able to obtain new parametrizations of various PDs of the matrix multiplication tensor and based on these parametrizations, we obtain more stable and faster practical FMM algorithms. Additionally, we find approximate PDs of lower rank by bounding the factor matrix elements in absolute value and choosing smart starting points. However, we are not able to obtain new exact PDs of lower rank and suspect that the algorithm converges to BRPDs for these ranks. It thus remains an open question if exact PDs of lower rank exist.

\section{Declarations}

\paragraph{Funding}
The first author was supported by the Fund for Scientific Research–Flanders (Belgium), EOS Project no 30468160.
The second author was supported by the Research Council KU Leuven, C1-project C14/17/073 and by the Fund for Scientific Research–Flanders (Belgium), EOS Project no 30468160 and project G0B0123N.

\paragraph{Code availability}
Our \Matlab code, e.g., the implementation of the AL method, used to obtain the results in this paper is publicly available.\footnote{URL: \url{https://github.com/CharlotteVermeylen/ALM_FMM_stability}} 

\paragraph{Authors' contributions}
C.V.\ wrote the main manuscript text and M.V.B.\ reviewed the manuscript.

\paragraph{Acknowledgments}
We would like to thank Panos Patrinos and Masoud Ahookhosk for their insightful discussions about the AL method and the LM method. 

\paragraph{Ethical Approval}
Not applicable.  

\paragraph{Availability of supporting data}
Not applicable.  

\paragraph{Competing interests}
Not applicable.






\bibliography{sn-bibliography}

\end{document}